\documentclass{zariski}

\title{A Foundation for Synthetic Algebraic Geometry}
\author{Felix Cherubini, Thierry Coquand and Matthias Hutzler}

\begin{document}

\maketitle

\begin{abstract}
  This is a foundation for algebraic geometry, developed internal to the Zariski topos, building on the work of Kock and Blechschmidt (\cite{kock-sdg}[I.12], \cite{ingo-thesis}).
  The Zariski topos consists of sheaves on the site opposite to the category of finitely presented algebras over a fixed ring, with the Zariski topology, i.e.\ generating covers are given by localization maps $A\to A_{f_1}$ for finitely many elements $f_1,\dots,f_n$ that generate the ideal $(1)=A\subseteq A$.
  We use homotopy type theory together with three axioms as the internal language of a (higher) Zariski topos.

  One of our main contributions is the use of higher types -- in the homotopical sense -- to define and reason about cohomology.
  Actually computing cohomology groups, seems to need a principle along the lines of our ``Zariski local choice'' axiom,
  which we justify as well as the other axioms using a cubical model of homotopy type theory.
\end{abstract}

\tableofcontents

\section*{Introduction}

Algebraic geometry is the study of solutions of polynomial equations using methods from geometry.
The central geometric objects in algebraic geometry are called \emph{schemes}.
Their basic building blocks are called \emph{affine schemes},
where, informally, an affine scheme corresponds to a solution sets of polynomial equations.
While this correspondence is clearly visible in the functorial approach to algebraic geometry and our synthetic approach,
it is somewhat obfuscated in the most commonly used, topological approach.

In recent years,
computer formalization of the intricate notion of affine schemes
received some attention as a benchmark problem
-- this is, however, \emph{not} a problem addressed by this work.
Instead, we use a synthetic approach to algebraic geometry,
very much alike to that of synthetic differential geometry.
This means, while a scheme in classical algebraic geometry is a complicated compound datum,
we work in a setting, based on homotopy type theory, where schemes are types,
with an additional property that can be defined within our synthetic theory.

Following ideas of Ingo Blechschmidt and Anders Kock  (\cite{ingo-thesis}, \cite{kock-sdg}[I.12]),
we use a base ring $R$ which is local and satisfies an axiom reminiscent of the Kock-Lawvere axiom.
This more general axiom is called \emph{synthetic quasi coherence (SQC)} by Blechschmidt and
a version quantifying over external algebras is called the \emph{comprehensive axiom}\footnote{
  In \cite{kock-sdg}[I.12], Kock's ``axiom $2_k$'' could equivalently be Theorem 12.2,
  which is exactly our synthetic quasi coherence axiom, except that it only quantifies over external algebras.
}
by Kock.
The exact concise form of the SQC axiom we use, was noted by David Jaz Myers in 2018 and communicated to the first author.

Before we state the SQC axiom, let us take a step back and look at the basic objects of study in algebraic geometry,
solutions of polynomial equations.
Given a system of polynomial equations
\begin{align*}
  p_1(X_1, \dots, X_n) &= 0\rlap{,} \\
  \vdots\quad\quad\;\;   \\
  p_m(X_1, \dots, X_n) &= 0\rlap{,}
\end{align*}
the solution set
$\{\, x : R^n \mid \forall i.\; p_i(x_1, \dots, x_n) = 0 \,\}$
is in canonical bijection to the set of $R$-algebra homomorphisms
\[ \Hom_{\Alg{R}}(R[X_1, \dots, X_n]/(p_1, \dots, p_m), R) \]
by identifying a solution $(x_1,\dots,x_n)$ with the homomorphism that maps each $X_i$ to $x_i$.
Conversely, for any $R$-algebra $A$ which is merely of the form $R[X_1, \dots, X_n]/(p_1, \dots, p_m)$,
we define the \emph{spectrum} of $A$ to be
\[
  \Spec A \colonequiv \Hom_{\Alg{R}}(A, R)
  \rlap{.}
\]
In contrast to classical, non-synthetic algebraic geometry,
where this set needs to be equipped with additional structure,
we postulate axioms that will ensure that $\Spec A$ has the expected geometric properties.
Namely, SQC is the statement that, for all finitely presented $R$-algebras $A$, the canonical map
  \begin{align*}
    A&\xrightarrow{\sim} (\Spec A\to R) \\
    a&\mapsto (\varphi\mapsto \varphi(a))
  \end{align*}
is an equivalence.
A prime example of a spectrum is $\A^1\colonequiv \Spec R[X]$,
which turns out to be the underlying set of $R$.
With the SQC axiom,
\emph{any} function $f:\A^1\to \A^1$ is given as a polynomial with coefficients in $R$.
In fact, all functions between affine schemes are given by polynomials.
Furthermore, for any affine scheme $\Spec A$,
the axiom ensures that
the algebra $A$ can be reconstructed as the algebra of functions $\Spec A \to R$,
therefore establishing a duality between affine schemes and algebras.

The Kock-Lawvere axiom used in synthetic differential geometry
might be stated as the SQC axiom restricted to (external) \emph{Weil-algebras},
whose spectra correspond to pointed infinitesimal spaces.
These spaces can be used in both synthetic differential and algebraic geometry
in very much the same way.

In the accompanying formalization \cite{formalization} of some basic results,
we use a setup which was already proposed by David Jaz Myers
in a conference talk (\cite{myers-talk1, myers-talk2}).
On top of Myers' ideas,
we were able to define schemes, develop some topological properties of schemes,
and construct projective space.

An important, not yet formalized result
is the construction of cohomology groups.
This is where the \emph{homotopy} type theory really comes to bear --
instead of the hopeless adaption of classical, non-constructive definitions of cohomology,
we make use of higher types,
for example the $n$-th Eilenberg-MacLane space $K(R,n)$ of the group $(R,+)$.
As an analogue of classical cohomology with values in the structure sheaf,
we then define cohomology with coefficients in the base ring as:
\[
  H^n(X,R):\equiv \propTrunc{X\to K(R,n)}_0
  \rlap{.}
\]
This definition is very convenient for proving abstract properties of cohomology.
For concrete calculations we make use of another axiom,
which we call \emph{Zariski-local choice}.
While this axiom was conceived of for exactly these kind of calculations,
it turned out to settle numerous questions with no apparent connection to cohomology.
One example is the equivalence of two notions of \emph{open subspace}.
A pointwise definition of openness was suggested to us by Ingo Blechschmidt and
is very convenient to work with.
However, classically, basic open subsets of an affine scheme are given
by functions on the scheme and the corresponding open is morally the collection of points where the function does not vanish.
With Zariski-local choice, we were able to show that these notions of openness agree in our setup.

Apart from SQC, locality of the base ring $R$ and Zariski-local choice,
we only use homotopy type theory, including univalent universes, truncations and some very basic higher inductive types.
Roughly, Zariski-local choice states that any surjection into an affine scheme merely has sections on a \emph{Zariski}-cover.%
\footnote{It is related to the set-theoretic axiom called
\emph{axiom of multiple choice} (AMC) \cite{vandenberg-moerdijk-amc} or \emph{weakly initial set of covers axiom} (WISC):
the set of all Zariski-covers of an affine scheme is weakly initial among all covers.
However, our axiom only applies to (affine) schemes, not all types or sets.}
The latter, internal, notion of cover corresponds quite directly to the covers in the site of the \emph{Zariski topos},
which we use to construct a model of homotopy type theory with our axioms.

More precisely, we can use the \emph{Zariski topos} over any base ring.
Toposes built using other Grothendieck topologies, like for example the étale topology, are not compatible with Zariski-local choice.
We did not explore whether an analogous setup can be used for derived algebraic geometry%
\footnote{Here, the word ``derived'' refers to the rings the algebraic geometry is built up from --
  instead of the 0-truncated rings we use, ``derived'' algebraic geometry would use simplicial or spectral rings.
  Sometimes, ``derived'' refers to homotopy types appearing in ``the other direction'', namely as the values of the sheaves that are used.
  In that direction, our theory is already derived, since we use homotopy type theory.
  Practically that means that we expect no problems when expanding our theory of synthetic schemes to what classic algebraic geometers
  call ``stacks''.
}
-- meaning that the 0-truncated rings we used are replaced by higher rings.
This is only because for a derived approach, we would have to work with higher monoids, which is currently infeasible
-- we are not aware of any obstructions for, say, an SQC axiom holding in derived algebraic geometry.

In total, the scope of our theory so far includes quasi-compact, quasi-separated schemes of finite type over an arbitrary ring.
These are all finiteness assumptions, that were chosen for convenience and include examples like closed subspaces of projective space,
which we want to study in future work, as example applications.
So far, we know that basic internal constructions, like affine schemes, correspond to the correct classical external constructions.
This can be expanded using our model, which is of course also important to ensure the consistency of our setup.

\section*{Formalization}
There is a related formalization project, which, at the time of writing,
contains the construction of projective $n$-space $\bP^n$ as a scheme.
The code may be found here:
\begin{center}
  \url{https://github.com/felixwellen/synthetic-geometry}
\end{center}
It makes extensive use of the algebra part of the cubical-agda library:
\begin{center}
  \url{https://github.com/agda/cubical}
\end{center}
-- which contains many contributions, in particular,
on finitely presented algebras and related concepts,
which where made in the scope of that project.

\section*{Acknowledgements}
We use work from Ingo Blechschmidt's PhD thesis, section 18 as a basis.
This includes in particular the synthetic quasi-coherence axiom and the assumption that the base ring is local.
David Jaz Myers had the idea to use Blechschmidt's ideas in homotopy type theory
and presented his ideas 2019 at the workshop ``Geometry in Modal Homotopy Type Theory'' in Pittsburgh.
Myers ideas include the algebra-setup we used in our formalization.

In December 2022, there was a mini-workshop in Augsburg, which helped with the development of this work.
We thank Jonas Höfer and Lukas Stoll for spotting a couple of small errors.

\section{Preliminaries}
\subsection{Subtypes and Logic}

We use the notation $\exists_{x:X}P(x)\colonequiv \propTrunc{\sum_{x:X}P(x)}$.
We use $+$ for the coproduct of types and for types $A,B$ we write
\[ A\vee B\colonequiv \propTrunc{ A+B }\rlap{.}\]

We will use subtypes extensively.

\begin{definition}
  \index{$\subseteq$}
  Let $X$ be a type.
  A \notion{subtype} of $X$ is a function $U:X\to\Prop$ to the type of propositions.
  We write $U\subseteq X$ to indicate that $U$ is as above.
  If $X$ is a set, a subtype may be called \notion{subset} for emphasis.
  For subtypes $A,B\subseteq X$, we write $A\subseteq B$ as a shorthand for pointwise implication.
\end{definition}

We will freely switch between subtypes $U:X\to\Prop$ and the corresponding embeddings
\[
  \begin{tikzcd}
    \sum_{x:X}U(x) \ar[r,hook] & X
  \end{tikzcd}
  \rlap{.}
\]
In particular, if we write $x:U$ for a subtype $U:X\to\Prop$, we mean that $x:\sum_{x:X}U(x)$ -- but we might silently project $x$ to $X$.

\begin{definition}
  Let $I$ and $X$ be types and $U_i:X\to\Prop$ a subtype for any $i:I$.
  \begin{enumerate}[(a)]
  \item The \notion{union} $\bigcup_{i:I}U_i$\index{$\bigcup_{i:I}U_i$} is the subtype $(x:X)\mapsto \exists_{i:I}U_i(x)$.
  \item The \notion{intersection} $\bigcap_{i:I}U_i$\index{$\bigcap_{i:I}U_i$} is the subtype $(x:X)\mapsto\prod_{i:I}U_i(x)$.
  \end{enumerate}
\end{definition}

We will use common notation for finite unions and intersections.
The following formula hold:

\begin{lemma}
  Let $I$, $X$ be types, $U_i:X\to\Prop$ a subtype for any $i:I$ and $V,W$ subtypes of $X$.
  \begin{enumerate}[(a)]
  \item Any subtype $P:V\to\Prop$ is a subtype of $X$ given by $(x:X)\mapsto\sum_{x:V}P(x)$.
  \item $V\cap \bigcup_{i:I} U_i=\bigcup (V\cap U_i)$.
  \item If $\bigcup_{i:I}U_i=X$ we have $V=\bigcup_{i:I}U_i\cap V$.
  \item If $\bigcup_{i:I}U_i=\emptyset$, then $U_i=\emptyset$ for all $i:I$.
  \end{enumerate}
\end{lemma}

\begin{definition}
  Let $X$ be a type.
  \begin{enumerate}[(a)]
  \item $\emptyset\colonequiv (x:X)\mapsto \emptyset$. \index{$\emptyset$}
  \item For $U\subseteq X$, let $\neg U\colonequiv (x:X)\mapsto \neg U(x)$. \index{$\neg U$}
  \item For $U\subseteq X$, let $\neg\neg U\colonequiv (x:X)\mapsto \neg\neg U(x)$. \index{$\neg\neg U$}
  \end{enumerate}
\end{definition}

\begin{lemma}
  $U=\emptyset$ if and only if $\neg\left(\exists_{x:X}U(x)\right)$.
\end{lemma}

\subsection{Homotopy type theory}

Our truncation levels start at $-2$, so $(-2)$-types are contractible, $(-1)$-types are propositions and $0$-types are sets.

\begin{definition}%
  Let $X$ and $I$ be types.
  A family of propositions $U_i:X\to\Prop$ \notion{covers} $X$,
  if for all $x:X$, there merely is a $i:I$ such that $U_i(x)$.
\end{definition}

\begin{lemma}%
  \label{kraus-glueing}
  Let $X$ and $I$ be types.
  For propositions $(U_i:X \to \Prop)_{i:I}$ that cover $X$ and $P:X\to \nType{0}$, we have the following glueing property: \\
  If for each $i:I$ there is a dependent function $s_i:(x:U_i)\to P(x)$ together with
  proofs of equality on intersections $p_{ij}:(x:U_i\cap U_j)\to (s_i(x)=s_j(x))$,
  then there is a globally defined dependent function $s:(x:X) \to P(x)$,
  such that for all $x:X$ and $i:I$ we have $U_i(x) \to s(x)=s_i(x)$
\end{lemma}

\begin{proof}
  We define $s$ pointwise.
  Let $x:X$.
  Using a Lemma of Kraus\footnote{For example this is the $n=-1$ case of \cite{dagstuhl-kraus}[Theorem 2.1].}
  and the $p_{ij}$, we get a factorization
  \[ \begin{tikzcd}[row sep=0mm]
    \sum_{i:I} U_i(x) \ar[rr, "s_{\pi_1(\_)}(x)"]\ar[rd] & & P(x) \\
    & \propTrunc{\sum_{i:I} U_i(x)}_{-1}\ar[ru,dashed] &
  \end{tikzcd} \]
-- which defines a unique value $s(x):P(x)$.
\end{proof}

Similarly we can prove.

\begin{lemma}%
  \label{kraus-glueing-1-type}
  Let $X$ and $I$ be types.
  For propositions $(U_i:X \to \Prop)_{i:I}$ that cover $X$ 
  and $P:X\to \nType{1}$, we have the following glueing property: \\
  If for each $i:I$ there is a dependent function $s_i:(x:U_i)\to P(x)$ together with
  proofs of equality on intersections $p_{ij}:(x:U_i\cap U_j)\to (s_i(x)=s_j(x))$ satisfying the cocycle
  condition $p_{ij}\cdot p_{jk} = p_{ik}$.
  then there is a globally defined dependent function $s:(x:X) \to P(x)$,
  such that for all $x:X$ and $i:I$ we have $p_i:U_i(x) \to s(x)=s_i(x)$ such that $p_i\cdot p_{ij} = p_j$.
\end{lemma}

This can be generalized to $\nType{k}$ for each {\em external} $k$.

The condition for $\nType{0}$ can be seen as an internal version of the usual patching {\em sheaf} condition.
The condition for $\nType{1}$ is then the internal version of the usual patching {\em $1$-stack} condition.

\subsection{Algebra}

\begin{definition}%
  \label{local-ring}
  A commutative ring $R$ is \notion{local} if $1\neq 0$ in $R$ and
  if for all $x,y:R$ such that $x+y$ is invertible, $x$ is invertible or $y$ is invertible.
\end{definition}

\begin{definition}%
  Let $R$ be a commutative ring.
  A \notion{finitely presented} $R$-algebra is an $R$-algebra $A$,
  such that there merely are natural numbers $n,m$ and polynomials $f_1,\dots,f_m:R[X_1,\dots,X_n]$
  and an equivalence of $R$-algebras $A\simeq R[X_1,\dots,X_n]/(f_1,\dots,f_m)$.
\end{definition}

\begin{definition}%
  \label{regular-element}
  Let $A$ be a commutative ring.
  An element $r:A$ is \notion{regular},
  if the multiplication map $r\cdot\_:A\to A$ is injective.
\end{definition}

\begin{lemma}%
  \label{units-products-regular}
  Let $A$ be a commutative ring.
  \begin{enumerate}[(a)]
  \item All units of $A$ are regular.
  \item If $f$ and $g$ are regular, their product $fg$ is regular.
  \end{enumerate}
\end{lemma}

\begin{example}
  The monomials $X^k:A[X]$ are regular.
\end{example}

\begin{lemma}%
  \label{polynomial-with-regular-value-is-regular}
  Let $f : A[X]$ be a polynomial
  and $a : A$ an element
  such that $f(a) : A$ is regular.
  Then $f$ is regular as an element of $A[X]$.
\end{lemma}

\begin{proof}
  After a variable substitution $X \mapsto X + a$
  we can assume that $f(0)$ is regular.
  Now let $g : A[X]$ be given with $fg = 0$.
  Then in particular $f(0) g(0) = 0$,
  so $g(0) = 0$.
  By induction,
  all coefficients of $g$ vanish.
\end{proof}

\begin{definition}
  Let $A$ be a ring and $f:A$.
  Then $A_{f}$ denotes the \notion{localization} of $A$ at $f$,
  i.e. a ring $A_f$ together with a homomorphism $A\to A_f$,
  such that for all homomorphisms $\varphi:A\to B$ such that
  $\varphi(f)$ is invertible, there is a unique homomorphism as indicated in the diagram:
  \begin{center}
    \begin{tikzcd}
      A\ar[r]\ar[rd,"\varphi",swap] & A_f\ar[d,dashed] \\
      & B
    \end{tikzcd}
    \rlap{.}
  \end{center}
  For $a:A$, we denote the image of $a$ in $A_f$ as $\frac{a}{1}$ and the inverse of $f$ as $\frac{1}{f}$.
\end{definition}

\begin{lemma}%
  \label{fg-ideal-local-global}
  Let $A$ be a commutative ring and $f_1,\dots,f_n:A$.
  For finitely generated ideals $I_i\subseteq A_{f_i}$,
  such that $A_{f_if_j}\cdot I_i=A_{f_if_j}\cdot I_j$ for all $i,j$,
  there is a finitely generated ideal $I\subseteq A$,
  such that $A_{f_i}\cdot I=I_i$ for all $i$.
\end{lemma}

\begin{proof}
  Choose generators 
  \[ \frac{g_{i1}}{1},\dots,\frac{g_{ik_i}}{1} \]
  for each $I_i$.
  These generators will still generate $I_i$, if we multiply any of them with any power of the unit $\frac{f_i}{1}$.
  Now
  \[ A_{f_if_j}\cdot I_i\subseteq A_{f_if_j}\cdot I_j \]
  means that for any $g_{ik}$, we have a relation
  \[ (f_if_j)^l g_{ik}=\sum_{l}h_{l}g_{jl}\]
  for some power $l$ and coefficients $h_{l}:A$.
  This means, that $f_i^lg_{ik}$ is contained in $I_j$.
  Multiplying $f_i^lg_{ik}$ with further powers of $f_i$ or multiplying $g_{jl}$ with powers of $f_j$ does not change that.
  So we can repeat this for all $i$ and $k$ to arrive at elements $\tilde{g_{ik}}:A$,
  which generate an ideal $I\subseteq A$ with the desired properties.
\end{proof}

The following definition also appears as \cite{ingo-thesis}[Definition 18.5]
and a version restricted to external finitely presented algebras was already used by Anders Kock in \cite{kock-sdg}[I.12]:

\begin{definition}
  \label{spec}
  The \notion{(synthetic) spectrum}\index{$\Spec A$} of a finitely presented $R$-algebra $A$
  is the set of $R$-algebra homomorphisms from $A$ to $R$:
  \[ \Spec A \colonequiv \Hom_{\Alg{R}}(A, R) \]
\end{definition}

We write $\A^n$ for $\Spec R[X_1, \dots, X_n]$,
which is canonically in bijection with $R^n$
by the universal property of the polynomial ring.
In particular,
$\A^1$ is (in bijection with) the underlying set of $R$.
Our convention is to use the letter $R$
when we think of it as an algebraic object,
and to write $\A^1$ (or $\A^n$) when we think of it as a set or a geometric object.

The $\Spec$ construction is functorial:

\begin{definition}
  \label{spec-on-maps}
  For an algebra homomorphism $f:\Hom_{\Alg{R}}(A,B)$
  between finitely presented $R$-algebras $A$ and $B$,
  we write \notion{$\Spec f$} for the map from $\Spec B$ to $\Spec A$
  given by precomposition with $f$.
\end{definition}

\begin{definition}%
  \label{standard-open-subset}
  Let $A$ be a finitely presented $R$-algebra.
  For $f:A$, the \notion{standard open subset} given by $f$,
  is the subtype 
  \[
    D(f)\colonequiv (x:\Spec A)\mapsto (x(f)\text{ is invertible})
    \rlap{.}
  \]
\end{definition}

later, we will use the following more general and related definitions:

\begin{definition}
  \label{open-closed-affine-subsets}
  Let $A$ be a finitely presented $R$-algebra.
  For $n:\N$ and $f_1,\dots,f_n:A$, there are
  \begin{enumerate}[(i)]
  \item the ``open'' subset
    \[
      D(f_1,\dots,f_n)\colonequiv (x:\Spec A)\mapsto (\text{$\exists_i$ such that $x(f_i)$ is invertible})
    \]  
  \item the ``closed'' subset
    \[
      V(f_1,\dots,f_n)\colonequiv (x:\Spec A)\mapsto (\forall_i\ x(f_i)=0)
    \]  
  \end{enumerate}
  It will be made precise in \Cref{topology-of-schemes}, in which sense these subsets are open or closed.
\end{definition}

We will later also need the notion of a \emph{Zariski-Cover} of a spectrum $\Spec A$,
for some finitely presented $R$-algebra $A$.
Intuitively, this is a collection of standard opens which jointly cover $\Spec A$.
Since it is more practical, we will however stay on the side of algebras.
A finite list of elements $f_1,\dots,f_n:A$ yields a Zariski-Cover,
if and only if they are a \emph{unimodular vector}:

\begin{definition}
  \label{unimodular}
  Let $A$ be a finitely presented $R$-algebra.
  Then a list $f_1,\dots,f_n:A$ of elements of $A$ is called \notion{unimodular}
  if we have an identity of ideals $(f_1,\dots,f_n)=(1)$.
  We use $\Um(A)$\index{$\Um(A)$} to denote the type of unimodular sequences in $A$:
  \[
    \Um(A)\colonequiv \sum_{n:\N}\sum_{f_1,\dots,f_n:A} (f_1,\dots,f_n)=(1)
    \rlap{.}
  \]
  We will sometimes drop the natural number and the equality and just write $(f_1,\dots,f_n):\Um(A)$.
\end{definition}

\begin{definition}
  $\AbGroup$\index{$\AbGroup$} denotes the type of abelian groups.
\end{definition}

\begin{lemma}%
  \label{surjective-abgroup-hom-is-cokernel}
  Let $A,B:\AbGroup$ and $f:A\to B$ be a homomorphism of abelian groups.
  Then $f$ is surjective, if and only if, it is a cokernel.
\end{lemma}

\begin{proof}
  A cokernel is a set-quotient by an effective relation,
  so the projection map is surjective.
  On the other hand, if $f$ is surjective and we are in the situation:
  \begin{center}
    \begin{tikzcd}
      \ker(f)\ar[r,hook]\ar[dr] & A\ar[r,"f",->>]\ar[dr,"g"] & B \\
      & 0\ar[r] & C
    \end{tikzcd}
  \end{center}
  then we can construct a map $\varphi:B\to C$ as follows.
  For $x:B$, we define the type of possible values $\varphi(x)$ in $C$ as
  \[
    \sum_{z:C}\exists_{y:A}(f(y)=x) \wedge g(y)=z
  \]
  which is a proposition by algebraic calculation.
  By surjectivity of $f$, this type is inhabited and therefore contractible.
  So we can define $\varphi(x)$ as its center of contraction.
\end{proof}

\ignore{
    - injective/embedding/-1-truncated map
  pushouts:
    - inclusions are jointly surjective,
    - pushouts of embeddings between sets are sets
  subtypes:
    - embeddings (composition, multiple definitions, relation to injection)  
    - we freely switch between predicates and types
    - subtypes of subtypes are subtypes
  pullbacks:
    - pasting (reference)
    - pullback of subtype = composition
  algebra:
    - free comm algebras, quotients
    - other definitions of polynomials
    - fp closed under: quotients, adjoining variables, tensor products
}

\section{Axioms}
\subsection{Statement of the axioms}%
\label{statement-of-axioms}

We always assume there is a fixed commutative ring $R$.
In addition, we assume the following three axioms about $R$,
which were already mentioned in the introduction,
but we will indicate which of these axioms are used to prove each statement
by listing their shorthands.

\begin{axiom}[Loc]%
  \label{loc}\index{Loc}
  $R$ is a local ring (\Cref{local-ring}).
\end{axiom}

\begin{axiom}[SQC]%
  \label{sqc}\index{sqc}
  For any finitely presented $R$-algebra $A$, the homomorphism
  \[ a \mapsto (\varphi\mapsto \varphi(a)) : A \to (\Spec A \to R)\]
  is an isomorphism of $R$-algebras.
\end{axiom}

\begin{axiom}[Z-choice]%
  \label{Z-choice}\index{Z-choice}
  Let $A$ be a finitely presented $R$-algebra
  and let $B : \Spec A \to \mU$ be a family of inhabited types.
  Then there merely exist unimodular $f_1, \dots, f_n : A$
  together with dependent functions $s_i : \Pi_{x : D(f_i)} B(x)$.
  As a formula\footnote{Using the notation from \Cref{unimodular}}:
  \[ (\Pi_{x : \Spec A} \propTrunc{B(x)}) \to
     \propTrunc{ ((f_1,\dots,f_n):\Um(A)) \times
      \Pi_i \Pi_{x : D(f_i)} B(x) }
     \rlap{.}
  \]
\end{axiom}

\subsection{First consequences}

Let us draw some first conclusions from the axiom (\axiomref{sqc}),
in combination with (\axiomref{loc}) where needed.

\begin{proposition}[using \axiomref{sqc}]%
  \label{spec-embedding}
  For all finitely presented $R$-algebras $A$ and $B$ we have an equivalence
  \[
    f\mapsto \Spec f : \Hom_{\Alg{R}}(A,B) = (\Spec B \to \Spec A)
    \rlap{.}
  \]
\end{proposition}

\begin{proof}
  By \Cref{algebra-from-affine-scheme}, we have a natural equivalence
  \[
    X\to \Spec (R^X)
  \]
  and by \axiomref{sqc}, the natural map
  \[
    A\to R^{\Spec A}
  \]
  is an equivalence.
  We therefore have a contravariant equivalence between
  the category of finitely presented $R$-algebras
  and the category of affine schemes.
  In particular, $\Spec$ is an embedding.
\end{proof}

An important consequence of \axiomref{sqc}, which may be called \notion{weak nullstellensatz}:

\begin{proposition}[using \axiomref{loc}, \axiomref{sqc}]%
  \label{weak-nullstellensatz}
  If $A$ is a finitely presented $R$-algebra,
  then we have $\Spec A=\emptyset$ if and only if $A=0$.
\end{proposition}

\begin{proof}
  If $\Spec A = \emptyset$
  then $A = R^{\Spec A} = R^\emptyset = 0$
  by (\axiomref{sqc}).
  If $A = 0$
  then there are no homomorphisms $A \to R$
  since $1 \neq 0$ in $R$ by (\axiomref{loc}).
\end{proof}

For example, this weak nullstellensatz suffices
to prove the following properties of the ring $R$,
which were already proven in
\cite{ingo-thesis}[Section 18.4].

\begin{proposition}[using \axiomref{loc}, \axiomref{sqc}]%
  \label{nilpotence-double-negation}\label{non-zero-invertible}\label{generalized-field-property}
  
  \begin{enumerate}[(a)]
  \item An element $x:R$ is invertible,
    if and only if $x\neq 0$.
  \item A vector $x:R^n$ is non-zero,
    if and only if one of its entries is invertible.
  \item An element $x:R$ is nilpotent,
    if and only if $\neg \neg (x=0)$.
  \end{enumerate}
\end{proposition}

\begin{proof}
  Part (a) is the special case $n = 1$ of (b).
  For (b),
  consider the $R$-algebra $A \colonequiv R/(x_1, \dots, x_n)$.
  Then the set $\Spec A \equiv \Hom_{\Alg{R}}(A, R)$
  is a proposition (that is, it has at most one element),
  and, more precisely, it is equivalent to the proposition $x = 0$.
  By \Cref{weak-nullstellensatz},
  the negation of this proposition is equivalent to $A = 0$
  and thus to $(x_1, \dots, x_n) = R$.
  Using (\axiomref{loc}),
  this is the case if and only if one of the $x_i$ is invertible.

  For (c),
  we instead consider the algebra $A \colonequiv R_x \equiv R[\frac{1}{x}]$.
  Here we have $A = 0$ if and only if $x$ is nilpotent,
  while $\Spec A$ is the proposition $\inv(x)$.
  Thus, we can finish by \Cref{weak-nullstellensatz},
  together with part (a) to go from $\lnot \inv(x)$ to $\lnot \lnot (x = 0)$.
\end{proof}

The following lemma,
which is a variant of \cite{ingo-thesis}[Proposition 18.32],
shows that $R$ is in a weak sense algebraically closed.
See \Cref{non-existence-of-roots} for a refutation of
a stronger formulation of algebraic closure of~$R$.

\begin{lemma}[using \axiomref{loc}, \axiomref{sqc}]%
  \label{polynomials-notnot-decompose}
  Let $f : R[X]$ be a polynomial.
  Then it is not not the case that:
  either $f = 0$ or
  $f = \alpha \cdot {(X - a_1)}^{e_1} \dots {(X - a_n)}^{e_n}$
  for some $\alpha : R^\times$,
  $e_i \geq 1$ and pairwise distinct $a_i : R$.
\end{lemma}

\begin{proof}
  Let $f : R[X]$ be given.
  Since our goal is a proposition,
  we can assume we have a bound $n$ on the degree of $f$,
  so
  \[ f = \sum_{i = 0}^n c_i X^i \rlap{.} \]
  Since our goal is even double-negation stable,
  we can assume $c_n = 0 \lor c_n \neq 0$
  and by induction $f = 0$ (in which case we are done)
  or $c_n \neq 0$.
  If $n = 0$ we are done,
  setting $\alpha \colonequiv c_0$.
  Otherwise,
  $f$ is not invertible (using $0 \neq 1$ by (\axiomref{loc})),
  so $R[X]/(f) \neq 0$,
  which by (\axiomref{sqc}) means that
  $\Spec(R[X]/(f)) = \{ x : R \mid f(x) = 0 \}$
  is not empty.
  Using the double-negation stability of our goal again,
  we can assume $f(a) = 0$ for some $a : R$
  and factor $f = (X - a_1) f_{n - 1}$.
  By induction, we get $f = \alpha \cdot (X - a_1) \dots (X - a_n)$.
  Finally, we decide each of the finitely many propositions $a_i = a_j$,
  which we can assume is possible
  because our goal is still double-negation stable,
  to get the desired form
  $f = \alpha \cdot {(X - \widetilde{a}_1)}^{e_1} \dots {(X - \widetilde{a}_n)}^{e_n}$
  with distinct $\widetilde{a}_i$.
\end{proof}

\section{Affine schemes}
\subsection{Affine-open subtypes}

We only talk about affine schemes of finite type, i.e. schemes of the form $\Spec A$ (\Cref{spec}),
where $A$ is a finitely presented algebra.

\begin{definition}%
  A type $X$ is \notion{(qc-)affine},
  if there is a finitely presented $R$-algebra $A$, such that $X=\Spec A$.
\end{definition}

If $X$ is affine, it is possible to reconstruct the algebra directly.

\begin{lemma}[using \axiomref{sqc}]%
  \label{algebra-from-affine-scheme}
  Let $X$ be an affine scheme, then there is a natural equivalence $X=\Spec (R^X)$.
\end{lemma}

\begin{proof}
  The natural map $X\to \Spec (R^X)$ is given by mapping $x:X$ to the
  evaluation homomorphism at $x$. 
  There merely is an $A$ such that $X=\Spec A$.
  Applying $\Spec$ to the canonical map $A\to R^{\Spec A}$,
  yields an equivalence by \axiomref{sqc}.
  This is a (one sided) inverse to the map above.
  So we have $X=\Spec (R^X)$.
\end{proof}

\begin{proposition}%
  Let $X$ be a type.
  The type of all finitely presented $R$-algebras $A$, such that $X=\Spec A$, is a proposition.
\end{proposition}

When we write ``$\Spec A$'' we implicitly assume $A$ is a finitely presented $R$-algebra.
Recall from \Cref{standard-open-subset}
that the standard open subset $D(f) \subseteq \Spec A$
is given by $D(f)(x)\colonequiv \inv(f(x))$.

\begin{example}[using \axiomref{loc}, \axiomref{sqc}]
  For $a_1, \dots, a_n : R$, we have
  \[ D((X - a_1) \cdots (X - a_n)) = \A^1 \setminus \{ a_1, \dots, a_n \} \rlap{.}\]
  Indeed,
  for any $x : \A^1$,
  $((X - a_1) \dots (X - a_n))(x)$ is invertible if and only if
  $x - a_i$ is invertible for all $i$.
  But by \Cref{non-zero-invertible}
  this means $x \neq a_i$ for all $i$.
\end{example}

\begin{definition}%
  \label{affine-open}
  Let $X=\Spec A$.
  A subtype $U:X\to\Prop$ is called \notion{affine-open},
  if one of the following logically equivalent statements holds:
  \begin{enumerate}[(i)]%
  \item $U$ is the union of finitely many affine standard opens.
  \item There are $f_1,\dots,f_n:A$ such that
    \[U(x) \Leftrightarrow \exists_{i} f_i(x)\neq 0 \]
  \end{enumerate}
\end{definition}

By \Cref{open-closed-affine-subsets} we have $D(f_1, \dots, f_n) = D(f_1) \cup \dots \cup D(f_n)$.
Note that in general, affine-open subtypes do not need to be affine
-- this is why we use the dash ``-''.

We will introduce a more general definition of open subtype in \Cref{qc-open}
and show in \Cref{qc-open-affine-open}, that the two notions agree on affine schemes.

\begin{proposition}
  Let $X = \Spec A$ and $f : A$.
  Then $D(f) = \Spec A[f^{-1}]$.
\end{proposition}

\begin{proof}
  \[ D(f) =
     \sum_{x : X} D(f)(x) =
     \sum_{x : \Spec A} \inv(f(x)) =
     \sum_{x : \Hom_{\Alg{R}}(A, R)} \inv(x(f)) =
     \Hom_{\Alg{R}}(A[f^{-1}], R) =
     \Spec A[f^{-1}]
     \]
\end{proof}

Affine-openness is transitive in the following sense:

\begin{lemma}%
  \label{affine-open-trans}
  Let $X=\Spec A$ and $D(f)\subseteq X$ be a standard open.
  Any affine-open subtype $U$ of $D(f)$ is also affine-open in $X$.
\end{lemma}

\begin{proof}
  It is enough to show the statement for $U=D(g)$, $g:A_f$.
  Then
  \[ g=\frac{h}{f^k}\rlap{.}\]
  Now $D(hf)$ is an affine-open in $X$,
  that coincides with $U$: \\
  Let $x:X$, then $(hf)(x)$ is invertible, if and only if both $h(x)$ and $f(x)$ are invertible.
  The latter means $x:D(f)$, so we can interpret $x$ as a homomorphism from $A_f$ to $R$.
  Then $x:D(g)$ means $x(g)$ is invertible, which is equivalent to $x(h)$ being invertible,
  since $x(f)^k$ is invertible anyway.
\end{proof}

\begin{lemma}[using \axiomref{loc}, \axiomref{sqc}]%
  \label{standard-open-empty}
  Let $X=\Spec A$ be an affine scheme and $D(f)\subseteq X$ a standard open,
  then $D(f)=\emptyset$, if and only if, $f$ is nilpotent.
\end{lemma}

\begin{proof}
  Since $D(f)=\Spec A_f$, by \Cref{weak-nullstellensatz}, we know $D(f)=\emptyset$,
  if and only if, $A_f=0$.
  The latter is equivalent to $f$ being nilpotent.
\end{proof}

More generally,
the Zariski-lattice consisting of the radicals
of finitely generated ideals of a finitely presented $R$-algebra $A$,
coincides with the lattice of open subtypes.
This means, that internal to the Zariski-topos,
it is not necessary to consider the full Zariski-lattice for a constructive treatment of schemes.

\begin{lemma}[using \axiomref{sqc}]%
  Let $A$ be a finitely presented $R$-algebra
  and let $f, g_1, \dots, g_n \in A$.
  Then we have $D(f) \subseteq D(g_1, \dots, g_n)$
  as subsets of $\Spec A$
  if and only if $f \in \sqrt{(g_1, \dots, g_n)}$.
\end{lemma}

\begin{proof}
  Since $D(g_1, \dots, g_n) = \{\, x \in \Spec A \mid x \notin V(g_1, \dots, g_n) \,\}$,
  \footnote{See \Cref{open-closed-affine-subsets} for ``$V(\dots)$''}
  the inclusion $D(f) \subseteq D(g_1, \dots, g_n)$
  can also be written as
  $D(f) \cap V(g_1, \dots, g_n) = \varnothing$, that is,
  $\Spec((A/(g_1, \dots, g_n))[f^{-1}]) = \varnothing$.
  By (\axiomref{sqc})
  this means that the finitely presented $R$-algebra $(A/(g_1, \dots, g_n))[f^{-1}]$
  is zero.
  And this is the case if and only if $f$ is nilpotent in $A/(g_1, \dots, g_n)$,
  that is, if $f \in \sqrt{(g_1, \dots, g_n)}$, as stated.
\end{proof}

In particular,
we have $\Spec A = \bigcup_{i = 1}^n D(f_i)$
if and only if $(f_1, \dots, f_n) = (1)$.

\subsection{Pullbacks of affine schemes}

\begin{lemma}%
  \label{affine-product}
  The product of two affine schemes is again an affine scheme,
  namely
  $\Spec A \times \Spec B = \Spec (A \otimes_R B)$.
\end{lemma}

\begin{proof}
  By the universal property of the tensor product $A \otimes_R B$.
\end{proof}

More generally we have:

\begin{lemma}[using \axiomref{sqc}]%
  \label{affine-fiber-product}
  Let $X=\Spec A,Y=\Spec B$ and $Z=\Spec C$ be affine schemes
  with maps $f:X\to Z$, $g:Y\to Z$.
  Then the pullback of this diagram is an affine scheme given by $\Spec (A\otimes_C B)$.
\end{lemma}

\begin{proof}
  The maps $f:X\to Z$, $g:Y\to Z$ are induced by $R$-algebra homomorphisms $f^*:A\to R$ and $g^*:B\to R$.
  Let
  \[ (h,k,p) : \Spec A \times_{\Spec C} \Spec B \]
  with $p:h\circ f^*=k\circ g^* $.
  This defines a $R$-cocone on the diagram
  \[
    \begin{tikzcd}
      A & C\ar[r,"g^*"]\ar[l,"f^*",swap] & B
    \end{tikzcd}
  \]
  Since $A\otimes_C B$ is a pushout in $R$-algebras,
  there is a unique $R$-algebra homomorphism $A\otimes_C B \to R$ corresponding to $(h,k,p)$.
\end{proof}

\subsection{Boundedness of functions to $\N$}

While the axiom \axiomref{sqc}
describes functions on an affine scheme
with values in $R$,
we can generalize it to functions taking values
in another finitely presented $R$-algebra,
as follows.

\begin{lemma}[using \axiomref{sqc}]%
  \label{algebra-valued-functions-on-affine}
  For finitely presented $R$-algebras $A$ and $B$,
  the function
  \begin{align*}
    A \otimes B &\xrightarrow{\sim} (\Spec A \to B) \\
    c &\mapsto (\varphi \mapsto (\varphi \otimes B)(c))
  \end{align*}
  is a bijection.
\end{lemma}

\begin{proof}
  We recall $\Spec (A \otimes B) = \Spec A \times \Spec B$
  from \Cref{affine-product}
  and calculate as follows.
  \begin{alignat*}{8}
    A \otimes B
    &=& (\Spec (A \otimes B) \to R)
    &=& (\Spec A \times \Spec B \to R)
    &=& (\Spec A \to (\Spec B \to R))
    &=& (\Spec A \to B) \\
    c
    &\mapsto& (\chi \mapsto \chi(c))
    &\mapsto& ((\varphi, \psi) \mapsto (\varphi \otimes \psi)(c))
    &\mapsto& (\varphi \mapsto (\psi \mapsto (\varphi \otimes \psi)(c)))
    &\mapsto& (\varphi \mapsto (\varphi \otimes B)(c))
  \end{alignat*}
  The last step is induced by the identification
  $B = (\Spec B \to R),\, b \mapsto (\psi \mapsto \psi(b))$,
  and we use the fact that
  $\psi \circ (\varphi \otimes B) = \varphi \otimes \psi$.
\end{proof}

\begin{lemma}[using \axiomref{sqc}]%
  \label{eventually-vanishing-sequence-on-affine}
  Let $A$ be a finitely presented $R$-algebra
  and let $s : \Spec A \to (\N \to R)$
  be a family of sequences,
  each of which eventually vanishes:
  \[ \prod_{x : \Spec A} \propTrunc{\sum_{N : \N} \prod_{n \geq N} s(x)(n) = 0} \]
  Then there merely exists one number $N : \N$
  such that $s(x)(n) = 0$ for all $x : \Spec A$ and all $n \geq N$.
\end{lemma}

\begin{proof}
  The set of eventually vanishing sequences $\N \to R$
  is in bijection with the set $R[X]$ of polynomials,
  by taking the entries of a sequence as the coefficients of a polynomial.
  So the family of sequences $s$
  is equivalently a family of polynomials $s : \Spec A \to R[X]$.
  Now we apply \Cref{algebra-valued-functions-on-affine} with $B = R[X]$
  to see that such a family corresponds to a polynomial $p : A[X]$.
  Note that for a point $x : \Spec A$,
  the homomorphism
  \[ x \otimes R[X] : A[X] = A \otimes R[X] \to R \otimes R[X] = R[X] \]
  simply applies the homomorphism $x$ to every coefficient of a polynomial,
  so we have $(s(x))_n = x(p_n)$.
  This concludes our argument,
  because the coefficients of $p$,
  just like any polynomial,
  form an eventually vanishing sequence.
\end{proof}

\begin{theorem}[using \axiomref{loc}, \axiomref{sqc}]%
  \label{boundedness}
  Let $A$ be a finitely presented $R$-algebra.
  Then every function $f : \Spec A \to \N$ is bounded:
  \[ \Pi_{f : \Spec A \to \N} \propTrunc{\Sigma_{N : \N} \Pi_{x : \Spec A} f(x) \le N}
     \rlap{.} \]
\end{theorem}

\begin{proof}
  Given a function $f : \Spec A \to \N$,
  we construct the family $s : \Spec A \to (\N \to R)$
  of eventually vanishing sequences
  given by
  \[
    s(x)(n) \colonequiv
    \begin{cases}
      \,1 &\text{if $n < f(x)$}\\
      \,0 &\text{else} \rlap{.}
    \end{cases}
  \]
  Since $0 \neq 1 : R$ by \axiomref{loc},
  we in fact have $s(x)(n) = 0$ if and only if $n \geq f(x)$.
  Then the claim follows from \Cref{eventually-vanishing-sequence-on-affine}.
\end{proof}

If we also assume the axiom \axiomref{Z-choice},
we can formulate the following simultaneous strengthening
of \Cref{eventually-vanishing-sequence-on-affine}
and \Cref{boundedness}.

\begin{proposition}[using \axiomref{loc}, \axiomref{sqc}, \axiomref{Z-choice}]%
  \label{strengthened-boundedness}
  Let $A$ be a finitely presented $R$-algebra.
  Let $P : \Spec A \to (\N \to \Prop)$
  be a family of upwards closed, merely inhabited subsets of $\N$.
  Then the set
  \[ \bigcap_{x : \Spec A} P(x) \subseteq \N \]
  is merely inhabited.
\end{proposition}

\begin{proof}
  By \axiomref{Z-choice},
  there merely exists a cover
  $\Spec A = \bigcup_{i = 1}^n D(f_i)$
  and functions $p_i : D(f_i) \to \N$
  such that $p_i(x) \in P(x)$ for all $x : D(f_i)$.
  By \Cref{boundedness},
  every $p_i : D(f_i) = \Spec A[f_i^{-1}] \to \N$
  is merely bounded by some $N_i : \N$,
  and then $\mathrm{max}(N_1, \dots, N_n) \in P(x)$ for all $x : \Spec A$.
\end{proof}

\section{Topology of schemes}%
\label{topology-of-schemes}

\subsection{Closed subtypes}

\begin{definition}%
  \label{closed-proposition}\label{closed-subtype}
  \begin{enumerate}[(a)]
  \item
    A \notion{closed proposition} is a proposition
    which is merely of the form $x_1 = 0 \land \dots \land x_n = 0$
    for some elements $x_1, \dots, x_n \in R$.
  \item
    Let $X$ be a type.
    A subtype $U : X \to \Prop$ is \notion{closed}
    if for all $x : X$, the proposition $U(x)$ is closed.
  \item
    For $A$ a finitely presented $R$-algebra
    and $f_1, \dots, f_n : A$,
    we set
    $V(f_1, \dots, f_n) \colonequiv
    \{\, x : \Spec A \mid f_1(x) = \dots = f_n(x) = 0 \,\}$.
  \end{enumerate}
\end{definition}

Note that $V(f_1, \dots, f_n) \subseteq \Spec A$ is a closed subtype
and we have $V(f_1, \dots, f_n) = \Spec (A/(f_1, \dots, f_n))$.

\begin{proposition}[using \axiomref{sqc}]%
  There is an order-reversing isomorphism of partial orders
  \begin{align*}
    \text{f.g.-ideals}(R) &\xrightarrow{{\sim}} \Omega_{cl} \\
    I &\mapsto (I = (0))
  \end{align*}
  between the partial order of finitely generated ideals of $R$
  and the partial order of closed propositions.
\end{proposition}

\begin{proof}
  For a finitely generated ideal $I = (x_1, \dots, x_n)$,
  the proposition $I = (0)$ is indeed a closed proposition,
  since it is equivalent to $x_1 = 0 \land \dots \land x_n = 0$.
  It is also evident that we get all closed propositions in this way.
  What remains to show is that
  \[ I = (0) \Rightarrow J = (0)
     \qquad\text{iff}\qquad
     J \subseteq I
     \rlap{\text{.}}
  \]
  For this we use synthetic quasicoherence.
  Note that the set $\Spec R/I = \Hom_{\Alg{R}}(R/I, R)$ is a proposition
  (has at most one element),
  namely it is equivalent to the proposition $I = (0)$.
  Similarly, $\Hom_{\Alg{R}}(R/J, R/I)$ is a proposition
  and equivalent to $J \subseteq I$.
  But then our claim is just the equation
  \[ \Hom(\Spec R/I, \Spec R/J) = \Hom_{\Alg{R}}(R/J, R/I) \]
  which holds by \Cref{spec-embedding},
  since $R/I$ and $R/J$ are finitely presented $R$-algebras
  if $I$ and $J$ are finitely generated ideals.
\end{proof}

\begin{lemma}[using \axiomref{sqc}]%
  \label{ideals-embed-into-closed-subsets}
  We have $V(f_1, \dots, f_n) \subseteq V(g_1, \dots, g_m)$
  as subsets of $\Spec A$
  if and only if
  $(g_1, \dots, g_m) \subseteq (f_1, \dots, f_n)$
  as ideals of $A$.
\end{lemma}

\begin{proof}
  The inclusion $V(f_1, \dots, f_n) \subseteq V(g_1, \dots, g_m)$
  means a map $\Spec (A/(f_1, \dots, f_n)) \to \Spec (A/(g_1, \dots, g_m))$
  over $\Spec A$.
  By \Cref{spec-embedding}, this is equivalent to
  a homomorphism $A/(g_1, \dots, g_m) \to A/(f_1, \dots, f_n)$,
  which in turn means the stated inclusion of ideals.
\end{proof}

\begin{lemma}[using \axiomref{loc}, \axiomref{sqc}, \axiomref{Z-choice}]%
  \label{closed-subtype-affine}
  A closed subtype $C$ of an affine scheme $X=\Spec A$ is an affine scheme
  with $C=\Spec (A/I)$ for a finitely generated ideal $I\subseteq A$.
\end{lemma}

\begin{proof}
  By \axiomref{Z-choice} and boundedness,
  there is a cover $D(f_1),\dots,D(f_l)$, such that
  on each $D(f_i)$, $C$ is the vanishing set of functions
  \[ g_1,\dots,g_n:D(f_i)\to R\rlap{.} \]
  By \Cref{ideals-embed-into-closed-subsets},
  the ideals generated by these functions
  agree in $A_{f_i f_j}$,
  so by \Cref{fg-ideal-local-global},
  there is a finitely generated ideal $I\subseteq A$,
  such that $A_{f_i}\cdot I$ is $(g_1,\dots,g_n)$
  and $C=\Spec A/I$.
\end{proof}

\subsection{Open subtypes}

While we usually drop the prefix ``qc'' in the definition below,
one should keep in mind, that we only use a definition of quasi compact open subsets.
The difference to general opens does not play a role so far,
since we also only consider quasi compact schemes later.

\begin{definition}%
  \label{qc-open}
  \begin{enumerate}[(a)]
  \item A proposition $P$ is \notion{(qc-)open}, if there merely are $f_1,\dots,f_n:R$,
    such that $P$ is equivalent to one of the $f_i$ being invertible.
  \item Let $X$ be a type.
    A subtype $U:X\to\Prop$ is \notion{(qc-)open}, if $U(x)$ is an open proposition for all $x:X$.
  \end{enumerate}
\end{definition}

\begin{proposition}[using \axiomref{loc}, \axiomref{sqc}]%
  \label{open-iff-negation-of-closed}
  A proposition $P$ is open
  if and only if
  it is the negation of some closed proposition
  (\Cref{closed-proposition}).
\end{proposition}

\begin{proof}
  Indeed, by \Cref{generalized-field-property},
  the proposition $\inv(f_1) \lor \dots \lor \inv(f_n)$
  is the negation of ${f_1 = 0} \land \dots \land {f_n = 0}$.
\end{proof}

\begin{proposition}[using \axiomref{loc}, \axiomref{sqc}]%
  \label{open-union-intersection}
  Let $X$ be a type.
  \begin{enumerate}[(a)]
  \item The empty subtype is open in $X$.
  \item $X$ is open in $X$.
  \item Finite intersections of open subtypes of $X$ are open subtypes of $X$.
  \item Finite unions of open subtypes of $X$ are open subtypes of $X$.
  \item Open subtypes are invariant under pointwise double-negation.
  \end{enumerate}
  Axioms are only needed for the last statement.
\end{proposition}

In \Cref{open-subscheme} we will see that open subtypes of open subtypes of a scheme are open in that scheme.
Which is equivalent to open propositions being closed under dependent sums.

\begin{proof}[of \Cref{open-union-intersection}]
  For unions, we can just append lists.
  For intersections, we note that invertibility of a product
  is equivalent to invertibility of both factors.
  Double-negation stability
  follows from \Cref{open-iff-negation-of-closed}.
\end{proof}

\begin{lemma}%
  \label{preimage-open}
  Let $f:X\to Y$ and $U:Y\to\Prop$ open,
  then the \notion{preimage} $U\circ f:X\to\Prop$ is open.
\end{lemma}

\begin{proof}
  If $U(y)$ is an open proposition for all $y : Y$,
  then $U(f(x))$ is an open proposition for all $x : X$.
\end{proof}

\begin{lemma}[using \axiomref{loc}, \axiomref{sqc}]%
  \label{open-inequality-subtype}
  Let $X$ be affine and $x:X$, then the proposition
  \[ x\neq y \]
  is open for all $y:X$.
\end{lemma}

\begin{proof}
  We show a proposition, so we can assume $\iota: X\to \A^n$ is a subtype.
  Then for $x,y:X$, $x\neq y$ is equivalent to $\iota(x)\neq\iota(y)$.
  But for $x,y:\A^n$, $x\neq y$ is the open proposition that $x-y\neq 0$.
\end{proof}

The intersection of all open neighborhoods of a point in an affine scheme,
is the formal neighborhood of the point.
We will see in \Cref{intersection-of-all-opens}, that this also holds for schemes.

\begin{lemma}[using \axiomref{loc}, \axiomref{sqc}]%
  \label{affine-intersection-of-all-opens}
  Let $X$ be affine and $x:X$, then the proposition
  \[ \prod_{U:X\to \Open}U(x)\to U(y) \]
  is equivalent to $\neg\neg (x=y)$.
\end{lemma}

\begin{proof}
  By \Cref{open-union-intersection}, $\neg\neg (x=y)$ implies $\prod_{U:X\to \Open}U(x)\to U(y)$.
  For the other implication,
  $\neg (x=y)$ is open by \Cref{open-inequality-subtype}, so we get a contradiction.
\end{proof}

We now show that our two definitions (\Cref{affine-open}, \Cref{qc-open})
of open subtypes of an affine scheme are equivalent.

\begin{theorem}[using \axiomref{loc}, \axiomref{sqc}, \axiomref{Z-choice}]%
  \label{qc-open-affine-open}
  Let $X=\Spec A$ and $U:X\to\Prop$ be an open subtype,
  then $U$ is affine open, i.e. there merely are $h_1,\dots,h_n:X\to R$ such that
  $U=D(h_1,\dots,h_n)$.
\end{theorem}

\begin{proof}
  Let $L(x)$ be the type of finite lists of elements of $R$,
  such that one of them being invertible is equivalent to $U(x)$.
  By assumption, we know
  \[\prod_{x:X}\propTrunc{L(x)}\rlap{.}\]
  So by \axiomref{Z-choice}, we have $s_i:\prod_{x:D(f_i)}L(x)$.
  We compose with the length function for lists to get functions $l_i:D(f_i)\to\N$.
  By \Cref{boundedness}, the $l_i$ are bounded.
  Since we are proving a proposition, we can assume we have actual bounds $b_i:\N$.
  So we get functions $\tilde{s_i}:D(f_i)\to R^{b_i}$,
  by append zeros to lists which are too short,
  i.e. $\widetilde{s}_i(x)$ is $s_i(x)$ with $b_i-l_i(x)$ zeros appended.

  Then one of the entries of $\widetilde{s}_i(x)$ being invertible,
  is still equivalent to $U(x)$.
  So if we define $g_{ij}(x)\colonequiv \pi_j(\widetilde{s}_i(x))$,
  we have functions on $D(f_i)$, such that
  \[
    D(g_{i1},\dots,g_{ib_i})=U\cap D(f_i)
    \rlap{.}
  \]
  By \Cref{affine-open-trans}, this is enough to solve the problem on all of $X$.
\end{proof}

This allows us to transfer one important lemma from affine-opens to qc-opens.
The subtlety of the following is that while it is clear that the intersection of two
qc-opens on a type, which are \emph{globally} defined is open again, it is not clear,
that the same holds, if one qc-open is only defined on the other.

\begin{lemma}[using \axiomref{loc}, \axiomref{sqc}, \axiomref{Z-choice}]%
  \label{qc-open-trans}
  Let $X$ be a scheme, $U\subseteq X$ qc-open in $X$ and $V\subseteq U$ qc-open in $U$,
  then $V$ is qc-open in $X$.
\end{lemma}

\begin{proof}
  Let $X_i=\Spec A_i$ be a finite affine cover of $X$.
  It is enough to show, that the restriction $V_i$ of $V$ to $X_i$ is qc-open.
  $U_i\colonequiv X_i\cap U$ is qc-open in $X_i$, since $X_i$ is qc-open.
  By \Cref{qc-open-affine-open}, $U_i$ is affine-open in $X_i$,
  so $U_i=D(f_1,\dots,f_n)$.
  $V_i\cap D(f_j)$ is affine-open in $D(f_j)$, so by \Cref{affine-open-trans},
  $V_i\cap D(f_j)$ is affine-open in $X_i$.
  This implies $V_i\cap D(f_j)$ is qc-open in $X_i$ and so is $V_i=\bigcup_{j}V_i\cap D(f_j)$.
\end{proof}

\begin{lemma}[using \axiomref{loc}, \axiomref{sqc}, \axiomref{Z-choice}]%
  \label{qc-open-sigma-closed}
  \begin{enumerate}[(a)]
  \item qc-open propositions are closed under dependent sums:
    if $P : \Open$ and $U : P \to \Open$,
    then the proposition $\sum_{x : P} U(x)$ is also open.
  \item Let $X$ be a type. Any open subtype of an open subtype of $X$ is an open subtype of $X$.
  \end{enumerate}
\end{lemma}

\begin{proof}
  \begin{enumerate}[(a)]
  \item Apply \Cref{qc-open-trans} to the point $\Spec R$.
  \item Apply the above pointwise.
  \end{enumerate}
\end{proof}

\begin{remark}
  \Cref{qc-open-sigma-closed} means that
  the (qc-) open propositions constitute a \notion{dominance}
  in the sense of~\cite{rosolini-phd-thesis}.
\end{remark}

The following fact about the interaction of closed and open propositions
is due to David Wärn.

\begin{lemma}%
  \label{implication-from-closed-to-open}
  Let $P$ and $Q$ be propositions
  with $P$ closed and $Q$ open.
  Then $P \to Q$ is equivalent to $\lnot P \lor Q$.
\end{lemma}

\begin{proof}
  We can assume $P = (f_1 = \dots = f_n = 0)$
  and $Q = (\inv(g_1) \lor \dots \lor \inv(g_m))$.
  Then we have:
  \begin{align*}
    (P \to Q) &= \qquad
    \text{\Cref{generalized-field-property} for $g_1, \dots, g_m$}\\
    (P \to \lnot (g_1 = \dots = g_m = 0)) &= \\
    \lnot (f_1 = \dots = f_n = g_1 = \dots = g_m = 0) &= \qquad
    \text{\Cref{generalized-field-property} for $f_1, \dots, f_n, g_1, \dots, g_m$}\\
    (\inv(f_1) \lor \dots \lor \inv(f_n) \lor \inv(g_1) \lor \dots \lor \inv(g_m) &= \qquad
    \text{\Cref{generalized-field-property} for $f_1, \dots, f_n$}\\
    \lnot P \lor Q &
  \end{align*}
\end{proof}

\section{Schemes}

\subsection{Definition of schemes}

In our internal setting, schemes are just types satisfying a property and morphisms of schemes are type theoretic functions.
The following definition \emph{does not} define schemes in general,
but something which is expected to correspond to quasi-compact, quasi separated schemes,
locally of finite type externally.

\begin{definition}%
  \label{schemes}
  A type $X$ is a \notion{(qc-)scheme}\index{scheme}
  if there merely is a cover by finitely many open subtypes $U_i:X\to\Prop$,
  such that each of the $U_i$ is affine.
\end{definition}

\begin{definition}
  \label{type-of-schemes}
  We denote the \notion{type of schemes} with $\Sch$\index{$\Sch$}.
\end{definition}

Zariski-choice \axiomref{Z-choice} extends to schemes:

\begin{proposition}[using \axiomref{Z-choice}]%
  \label{zariski-choice-scheme}
  Let $X$ be a scheme and $P:X\to \Type$ with $\prod_{x:X}\propTrunc{P(x)}$,
  then there merely is a cover $U_i$ by standard opens of the affine parts of $X$,
  such that there are $s_i:\prod_{x:U_i}P(x)$ for all $i$.
\end{proposition}

\subsection{General Properties}

\begin{lemma}[using \axiomref{loc}, \axiomref{sqc}]%
  \label{intersection-of-all-opens}
  Let $X$ be a scheme and $x:X$, then for all $y:X$ the proposition
  \[ \prod_{U:X\to \Open}U(x)\to U(y) \]
  is equivalent to $\neg\neg (x=y)$.
\end{lemma}

\begin{proof}
  By \Cref{open-union-intersection},
  open proposition are always double-negation stable,
  which settles one implication.
  For the implication
  \[ \left(\prod_{U:X\to \Open}U(x)\to U(y)\right) \Rightarrow \neg\neg (x=y) \]
  we can assume that $x$ and $y$ are both inside an open affine $U$
  and use that the statement holds for affine schemes by \Cref{affine-intersection-of-all-opens}.
\end{proof}

\subsection{Glueing}

\begin{proposition}[using \axiomref{loc}, \axiomref{sqc}, \axiomref{Z-choice}]%
  Let $X,Y$ be schemes and $f:U\to X$, $g:U\to Y$ be embeddings with open images in $X$ and $Y$,
  then the pushout of $f$ and $g$ is a scheme.
\end{proposition}

\begin{proof}
  As is shown for example \href{https://github.com/agda/cubical/blob/310a0956bb45ea49a5f0aede0e10245292ae41e0/Cubical/HITs/Pushout/KrausVonRaumer.agda#L170-L176}{here},
  such a pushout is always 0-truncated.
  Let $U_1,\dots,U_n$ be a cover of $X$ and $V_1,\dots,V_m$ be a cover of $Y$.
  By \Cref{qc-open-trans}, $U_i\cap U$ is open in $Y$,
  so we can use (large) pushout-recursion to construct a subtype $\tilde{U_i}$,
  which is open in the pushout and restricts to $U_i$ on $X$ and $U_i\cap U$ on $Y$.
  Symmetrically we define $\tilde{V_i}$ and in total get an open finite cover of the pushout.
  The pieces of this new cover are equivalent to their counterparts in the covers of $X$ and $Y$,
  so they are affine as well.
\end{proof}

\subsection{Subschemes}

\begin{definition}
  Let $X$ be a scheme.
  A \notion{subscheme} of $X$ is a subtype $Y:X\to\Prop$,
  such that $\sum Y$ is a scheme.
\end{definition}

\begin{proposition}[using \axiomref{loc}, \axiomref{sqc}, \axiomref{Z-choice}]%
  \label{open-subscheme}
  Any open subtype of a scheme is a scheme.
\end{proposition}

\begin{proof}
  Using \Cref{qc-open-affine-open}.
\end{proof}

\begin{proposition}[using \axiomref{loc}, \axiomref{sqc}, \axiomref{Z-choice}]%
  \label{closed-subscheme}
  Any closed subtype $A:X\to \Prop$ of a scheme $X$ is a scheme.
\end{proposition}

\begin{proof}
  Any open subtype of $X$ is also open in $A$.
  So it is enough to show,
  that any affine open $U_i$ of $X$,
  has affine intersection with $A$.
  But $U_i\cap A$ is closed in $U_i$ and therefore affine by \Cref{closed-subtype-affine}.
\end{proof}

\subsection{Equality types}

\begin{lemma}%
  \label{affine-equality-closed}
  Let $X$ be an affine scheme and $x,y:X$,
  then $x=_Xy$ is an affine scheme
  and $((x,y):X\times X)\mapsto x=_Xy$
  is a closed subtype of $X\times X$.
\end{lemma}

\begin{proof}
  Any affine scheme is merely embedded into $\A^n$ for some $n:\N$.
  The proposition $x=y$ for elements $x,y:\A^n$ is equivalent to $x-y=0$,
  which is equivalent to all entries of this vector being zero.
  The latter is a closed proposition.
\end{proof}

\begin{proposition}[using \axiomref{loc}, \axiomref{sqc}, \axiomref{Z-choice}]%
  \label{equality-scheme}
  Let $X$ be a scheme.
  The equality type $x=_Xy$ is a scheme for all $x,y:X$.
\end{proposition}

\begin{proof}
  Let $x,y:X$ and
  $U\subseteq X$ be an affine open containing $x$.
  Then $U(y)\wedge x=y$ is equivalent to $x=y$, so it is enough to show that $U(y)\wedge x=y$ is a scheme.
  As a open subscheme of the point, $U(y)$ is a scheme and $(x:U(y))\mapsto x=y$ defines a closed subtype by \Cref{affine-equality-closed}.
  But this closed subtype is a scheme by \Cref{closed-subscheme}.
\end{proof}

\subsection{Dependent sums}

\begin{theorem}[using \axiomref{loc}, \axiomref{sqc}, \axiomref{Z-choice}]%
  \label{sigma-scheme}
  Let $X$ be a scheme and for any $x:X$, let $Y_x$ be a scheme.
  Then the dependent sum
  \[ \left((x:X)\times Y_x\right)\equiv \sum_{x:X}Y_x\]
  is a scheme.
\end{theorem}

\begin{proof}
  We start with an affine $X=\Spec A$ and $Y_x=\Spec B_x$.
  Locally on $U_i = D(f_i)$, for a Zariski-cover $f_1,\dots,f_l$ of $X$,
  we have $B_x=\Spec R[X_1,\dots,X_{n_i}]/(g_{i,x,1},\dots,g_{i,x,m_i})$
  with polynomials $g_{i,x,j}$.
  In other words, $B_x$ is the closed subtype of $\A^{n_i}$
  where the functions $g_{i,x,1},\dots,g_{i,x,m_i}$ vanish.
  By \Cref{affine-fiber-product}, the product
  \[ V_i\colonequiv U_i\times \A^{n_i}\]
  is affine.
  The type $(x:U_i)\times \Spec B_x\subseteq V_i$ is affine,
  since it is the zero set of the functions
  \[ ((x,y):V_i)\mapsto g_{i,x,j}(y) \]
  Furthermore, $W_i\colonequiv (x:U_i)\times \Spec B_x$
  is open in $(x:X)\times Y_x$,
  since $W_i(x)$ is equivalent to $U_i(\pi_1(x))$,
  which is an open proposition.

  This settles the affine case.
  We will now assume, that
  $X$ and all $Y_x$ are general schemes.
  We pass again to a cover of $X$ by affine open $U_1,\dots,U_n$.
  We can choose the latter cover,
  such that for each $i$ and $x:U_i$, the $Y_{\pi_1(x)}$
  are covered by $l_i$ many open affine pieces $V_{i,x,1},\dots,V_{i,x,l_i}$
  (by \Cref{boundedness}).
  Then $W_{i,j}\colonequiv(x:U_i)\times V_{i,x,j}$ is affine by what we established above.
  It is also open.
  To see this, let $(x,y):((x:X)\times Y_x)$.
  We want to show, that $(x,y)$ being in $W_{i,j}$ is an open proposition.
  We have to be a bit careful, since the open proposition
  $V_{i,x,j}$ is only defined, for $x:U_i$.
  So the proposition we are after is $(z:U_i(x,y))\times V_{i,z,j}(y)$.
  But this proposition is open by \Cref{qc-open-sigma-closed}.
\end{proof}

It can be shown, that if $X$ is affine and for $Y:X\to\Sch$,
$Y_x$ is affine for all $x:X$,
then $(x:X)\times Y_x$ is affine.
An easy proof using cohomology is \href{https://felix-cherubini.de/random.pdf}{here}. 

\begin{corollary}
  \label{scheme-map-classification}
  Let $X$ be a scheme.
  For any other scheme $Y$ and any map $f:Y\to X$,
  the fiber map
  $(x:X)\mapsto \fib_f(x)$
  has values in the type of schemes $\Sch$.
  Mapping maps of schemes to their fiber maps,
  is an equivalence of types
  \[ \left(\sum_{Y:\Sch}(Y\to X)\right)\simeq (X\to \Sch)\rlap{.}\]
\end{corollary}

\begin{proof}
  By univalence, there is an equivalence
  \[ \left(\sum_{Y:\Type}(Y\to X)\right)\simeq (X\to \Type)\rlap{.} \]
  From left to right, the equivalence is given by turning a $f:Y\to X$ into $x\mapsto \fib_f(x)$,
  from right to left is given by taking the dependent sum.
  So we just have to note, that both constructions preserve schemes.
  From left to right, this is \Cref{fiber-product-scheme}, from right to left,
  this is \Cref{sigma-scheme}.
\end{proof}

Subschemes are classified by propositional schemes:

\begin{corollary}
  Let $X$ be a scheme.
  $Y:X\to\Prop$ is a subscheme,
  if and only if $Y_x$ is a scheme for all $x:X$.
\end{corollary}

\begin{proof}
  Restriction of \Cref{scheme-map-classification}.
\end{proof}

We will conclude now,
that the pullback of a cospan of schemes is a scheme.

\begin{theorem}[using \axiomref{loc}, \axiomref{sqc}, \axiomref{Z-choice}]%
  \label{fiber-product-scheme}
  Let
  \[
    \begin{tikzcd}
      X\ar[r,"f"] & Z & Y\ar[l,swap,"g"]
    \end{tikzcd}
  \]
  be schemes, then the \notion{pullback} $X\times_Z Y$ is also a scheme.
\end{theorem}

\begin{proof}
  The type $X\times_Z Y$ is given as the following iterated dependent sum:
  \[ \sum_{x:X}\sum_{y:Y}f(x)=g(y)\rlap{.}\]
  The innermost type, $f(x)=g(y)$
  is the equality type in the scheme $Z$ and by \Cref{equality-scheme} a scheme.
  By applying \Cref{sigma-scheme} twice, we prove that the iterated dependent sum is a scheme.
\end{proof}

\section{Projective space}

\subsection{Construction of projective spaces}
We give two definitions of projective space, which differ only in size.
First, we will define $n$-dimensional projective space,
as the type of lines in a $(n+1)$-dimensional vector space $V$.
This gives a good mapping-in property --
maps from a type $X$ into projective space are then just families of lines in $V$ on $X$.
Or in the words of traditional algebraic geometry:
projective $n$-space is a fine moduli space for lines in $V$.

The second construction is closer to what can be found in a typical
introductory textbook on algebraic geometry
(see for example \cite[Section I.2]{Hartshorne}),
i.e.\ projective $n$-space is constructed as a quotient of $\A^{n+1}\setminus\{0\}$.
We will show that this quotient is a scheme,
again analogous to what can be found in textbooks.
In both, construction and proof,
we do not have to pass to an algebraic representation
and can work directly with the types of interest.
Finally, in \Cref{space-of-lines-is-projective-space}
we show that the two constructions are equivalent.

\begin{definition}%
  \begin{enumerate}[(a)]
  \item An $n$-dimensional $R$-\notion{vector space} is an $R$-module $V$,
    such that $\propTrunc{ V = R^n }$. 
  \item We write $\Vect{R}{n}$ for the type of these vector spaces and $V\setminus\{0\}$ for the type
    \[ \sum_{x:V}x\neq 0\]
  \item A \notion{vector bundle} on a type $X$ is a map $V:X\to \Vect{R}{n}$. 
  \end{enumerate}
\end{definition}

The following defines projective space as the space of lines in a vector space.
This is a large type.
We will see below, that the second, equivalent definition is small.

\begin{definition}%
  \label{projective-space-as-space-of-lines}
  \begin{enumerate}[(a)]
  \item A \notion{line} in an $R$-vector space $V$ is a subtype $L:V\to \Prop$,
    such that there exists an $x:V\setminus\{0\}$ with
    \[ \prod_{y:V}\left(L (y) \Leftrightarrow \exists c:R.y=c\cdot x\right)\]
  \item The space of all lines in a fixed $n$-dimensional vector space $V$ is the \notion{projectivization} of $V$:
    \[ \bP(V)\colonequiv \sum_{L:V\to \Prop} L \text{ is a line}  \]
  \item \notion{Projective $n$-space} $\bP^n \colonequiv \bP(\A^{n+1})$ is the projectivization of $\A^{n+1}$.
  \end{enumerate}
\end{definition}

\begin{proposition}%
  \label{lines-are-one-dimensional}
  For any vector space $V$ and line $L\subseteq V$,
  $L$ is 1-dimensional in the sense that $\propTrunc{L=_{\Mod{R}}R}$.
\end{proposition}

\begin{proof}
  Let $L$ be a line.
  We merely have $x:V\setminus\{0\}$ such that 
  \[ \prod_{y:V}\left(L (y) \Leftrightarrow \exists c:R.y=c\cdot x\right)\]
  We may replace the ``$\exists$'' with a ``$\sum$'',
  since $c$ is uniquely determined for any $x,y$.
  This means we can construct the map $\alpha\mapsto \alpha\cdot x:R\to L$ and it is an equivalence.
\end{proof}

We now give the small construction:

\begin{definition}[using \axiomref{loc}, \axiomref{sqc}]%
  \label{projective-space-hit}
  Let $n:\N$.
  \notion{Projective $n$-space}\index{$\bP^n$} $\bP^n$ is the set quotient of the type $\A^{n+1}\setminus\{0\}$ by the relation
  \[
    x \sim y \colonequiv \sum_{\lambda : R} \lambda x=y\rlap{.}
  \]
  By \Cref{generalized-field-property}, the non-zero vector $y$ has an invertible entry,
  so that the right hand side is a proposition and $\lambda$ is a unit.
  We write $[x_0:\dots:x_n]:\bP^n$ for the equivalence class of $(x_0,\dots,x_n):\A^{n+1}\setminus\{0\}$.
\end{definition}

\begin{theorem}[using \axiomref{loc}, \axiomref{sqc}]%
  \label{projective-space-is-scheme}
  $\bP^n$ is a scheme.
\end{theorem}

\begin{proof}
  Let $U_i([x_0:\dots:x_n])\colonequiv (x_i\neq 0)$.
  This is well-defined, since the proposition is invariant under multiplication by a unit.
  Furthermore, $U_i$ is open and the $U_i$ form a cover,
  by the generalized field property
  (\Cref{generalized-field-property}).

  So what remains to be shown, is that the $U_i$ are affine.
  We will show that $U_i=\A^n$.
  As an intermediate step, we have:
  \[
    U_i=\{(x_0,\dots,x_n):\A^{n+1}\mid x_i=1\}
  \]
  by mapping $[x_0:\dots:x_n]$ with $x_i\neq 0$
  to $\left(\frac{x_0}{x_i},\dots,\frac{x_n}{x_i}\right)$
  and conversely, $(x_0,\dots,x_n)$ with $x_i=1$ to $[x_0:\dots:x_{i-1}:1:x_{i+1}:\dots:x_n]\in U_i$.

  But then, $\{(x_0,\dots,x_n):\A^{n+1}\mid x_i=1\}$
  is equivalent to $\A^n$ by leaving out the $i$-th component,
  so the $U_i$ are affine.
\end{proof}

To conclude with the constructions of projective space,
we show that our two constructions are equivalent:

\begin{proposition}[using \axiomref{loc}, \axiomref{sqc}]
  \label{space-of-lines-is-projective-space}
  For all $n:\N$, the scheme $\bP^n$ as defined in \Cref{projective-space-hit},
  is equivalent to $\bP(\A^{n+1})$ as defined in \Cref{projective-space-as-space-of-lines}.
\end{proposition}

\begin{proof}
  Let $\varphi:\bP^n\to \{\text{lines in $\A^{n+1}$}\}$
  be given by mapping $[x_0:\dots:x_n]$ to $\langle (x_0,\dots,x_n)\rangle\subseteq \A^{n+1}$,
  i.e.\ the line generated by the vector $x\colonequiv (x_0,\dots,x_n)$.
  The map is well-defined, since multiples of $x$ generate the same line.

  Then $\varphi$ is surjective,
  since for any line $L\subseteq \A^{n+1}$,
  there merely is a non-zero $x\in L$,
  that we can take as a preimage.
  To conclude, we note that $\varphi$ is also an embedding.
  So let $\varphi([x])=\varphi([y])$.
  Then, since $\langle x\rangle=\langle y\rangle$, there is a $\lambda\in R^\times$,
  such that $x=\lambda y$, so $[x]=[y]$.
\end{proof}

Let us prove some basic facts about equality of points in $\bP^n$.

\begin{lemma}[using \axiomref{loc}, \axiomref{sqc}]
  \label{equality-2-by-2-minor}
  For two points $[x_0:\dots:x_n],[y_0:\dots:y_n]:\bP^n$ we have:
  \[
    [x]=[y] \Leftrightarrow \prod_{i\neq j}x_iy_j=y_ix_j
    \rlap{.}
  \]
  And dually:
  \[
    [x]\neq [y] \Leftrightarrow \bigvee_{i\neq j}x_iy_j\neq y_ix_j
    \rlap{.}
  \]
  As a consequence, $[x]=[y]$ is closed and $[x]\neq [y]$ is open.
\end{lemma}

\begin{proof}
  $[x]$ and $[y]$ are equal,
  if and only if there merely is a $\lambda:R^\times$,
  such that $\lambda x = y$.
  By calculation, if there is such a $\lambda$,
  we always have $x_iy_j=y_ix_j$.

  So let $x_iy_j=y_ix_j$ for all $i\neq j$.
  Then, in particular, there are $i,j$ such that $x_i\neq 0$ and $y_j\neq 0$.
  If $i=j$, we define $\lambda \colonequiv \frac{x_i}{y_i}$.
  If $i\neq j$, we have $x_iy_j=y_ix_j$ and therefore $y_i\neq 0$ and $x_j\neq 0$,
  so we can also set $\lambda \colonequiv \frac{x_i}{y_i}$.
  By calculation, we have $\lambda y=x$.

  The dual statement follows by \Cref{generalized-field-property}.
\end{proof}

\begin{lemma}[using \axiomref{loc}, \axiomref{sqc}]%
  \label{projective-space-apartness-relation}
  Inequality of points of $\bP^n$ is an \notion{apartness relation}.
  That means the following holds:
  \begin{enumerate}[(i)]
  \item $\forall x:\bP^n.\; \lnot (x\neq x)$.
  \item $\forall x,y:\bP^n.\; x\neq y\Rightarrow y\neq x$.
  \item If $x\neq y$, we have $\forall z:\bP^n.\; x\neq z \vee z\neq y$.
  \end{enumerate}
\end{lemma}

\begin{proof}
  The first two statements hold in general for inequality.
  For the third statement, let $x,y,z:\bP^n$.
  Note that if $x=z$ and $z=y$, it follows that $x=y$.
  So we have $\neg (x=y)\Rightarrow \neg (x=z\wedge z=y)$.
  By \Cref{equality-2-by-2-minor}, $x=y$ and $x=z\wedge z=y$
  are both equivalent to the statement that some vector with components in $R$ is zero,
  so we can replace negated equality, with existence of a non-zero element,
  or more explicitly, the following are equivalent:
  \begin{align*}
    &\neg (x=y)\Rightarrow \neg (x=z\wedge z=y) \\
    &\neg \left(\prod_{i\neq j}x_iy_j=y_ix_j\right)
       \Rightarrow \neg \left(\prod_{i\neq j}x_iz_j=z_ix_j \wedge \prod_{i\neq j}y_iz_j=z_iy_j \right) \\
    & \left(\bigvee_{i\neq j}x_iy_j\neq y_ix_j\right) \Rightarrow \left(\bigvee_{i\neq j}x_iz_j\neq z_ix_j
       \vee \bigvee_{i\neq j}z_iy_j\neq y_iz_j\right) \\
    & (x\neq y) \Rightarrow (x\neq z) \vee (z\neq y)
  \end{align*}
\end{proof}

\begin{example}[using \axiomref{loc}, \axiomref{sqc}]
  Let $s:\bP^1\to \bP^1$ be given by $s([x:y])\colonequiv [x^2:y^2]$
  (see \Cref{projective-space-hit} for notation).
  Let us compute some fibers of $s$. The fiber $\fib_s([0:1])$ is
  by definition the type
  \[
    \sum_{[x:y]:\bP^1}[x^2:y^2]=[0:1]\rlap{.}
  \]
  So for any $x:R$ with $x^2=0$, $[x:1]:\fib_s([0:1])$  and
  any other point $(x,y)$ such that $[x:y]$ is in $\fib_s([0:1])$,
  already yields an equivalent point, since $y$ has to be invertible.

  This shows that the fiber over $[0:1]$ is a first order disk, i.e. $\D(1)=\{x:R|x^2=0\}$.
  The same applies to the point $[1:0]$.
  To analyze $\fib_s([1:1])$, let us assume $2\neq 0$ (in $R$).
  Then we know, the two points $[1:-1]$ and $[1:1]$ are in $\fib_s([1:1])$ and they are different.
  It will turn out, that any point in $\fib_s([1:1])$ is equal to one of those two.
  For any $[x':y']:\fib_s([1:1])$, we can assume $[x':y']=[x:1]$ and $x^2=1$, or equivalently $(x-1)(x+1)=0$.
  By \Cref{projective-space-apartness-relation}, inequality in $\bP^n$ is an apartness relation.
  So for each $x:R$, we know $x-1$ is invertible or $x+1$ is invertible.
  But this means that for any $x:R$ with $(x-1)(x+1)=0$, that $x=1$ or $x=-1$.

  While the fibers are not the same in general,
  they are all affine and have the same size in the sense that for each $\Spec A_x\colonequiv \fib_s(x)$,
  we have that $A_x$ is free of rank 2 as an $R$-module.
  To see this, let us first note,
  that $\fib_s([x:y])$ is completely contained in an affine subset of $\bP^1$.
  This is a proposition, so we can use that either $x$ or $y$ is invertible.
  Let us assume without loss of generality, that $y$ is invertible,
  then
  \[
    \fib_s([x:y])=\fib_s([\frac{x}{y}:1])
    \rlap{.}
  \]
  The second component of each element in the fiber has to be invertible,
  so it is contained in an affine subset, which we identify with $\A^1$.
  Let us rewrite with $z\colonequiv \frac{x}{y}$.
  Then
  \[
    \fib_s([z:1])=\sum_{a:\A^1}(a^2=z)=\Spec R[X]/(X^2-z)
  \]
  and $R[X]/(X^2-z)$ is free of rank 2 as an $R$-module.
\end{example}

\subsection{Functions on $\bP^n$}

Here we prove the classical fact that all functions $\bP^n\to R$ are constant.
We start with the case $n = 1$.

\begin{lemma}[using \axiomref{loc}, \axiomref{sqc}]%
  \label{functions-on-projective-line-constant}
  All functions $\bP^1 \to R$ are constant.
\end{lemma}

\begin{proof}
  Consider the affine cover of $\bP^1 = U_0 \cup U_1$
  as in the proof of \Cref{projective-space-is-scheme}.
  Both $U_0$ and $U_1$ are isomorphic to $\A^1$
  and the intersection $U_0 \cap U_1$ is $\A^1 \setminus \{0\}$,
  embedded in $U_0$ by $x \mapsto x$
  and in $U_1$ by $x \mapsto \frac{1}{x}$.
  So we have a pushout square as follows.
  \[ \begin{tikzcd}
    \A^1 \setminus \{0\} \ar[r] \ar[d]
    \ar[dr, phantom, very near end, "\ulcorner"] &
    \A^1 \ar[d] \\
    \A^1 \ar[r] &
    \bP^1
  \end{tikzcd} \]
  If we apply the functor $X \mapsto R^X$ to this diagram,
  we obtain a pullback square of $R$ algebras,
  and we can insert the known $R$ algebras for the affine schemes involved.
  \[ \begin{tikzcd}
    R[X, Y]/(1 - XY) &
    R[Y] \ar[l] \\
    R[X] \ar[u] &
    R^{\bP^1} \ar[l] \ar[u]
    \ar[lu, phantom, very near start, "\ulcorner"]
  \end{tikzcd} \]
  Here, the different variable names $X$ and $Y$ indicate
  the resulting homomorphisms.
  Now it is an algebraic computation,
  understanding the elements of $R[X, Y]/(1 - XY)$ as Laurent polynomials,
  to see that the pullback is the algebra $R$,
  so we have $R^{\bP^1} = R$ as desired.
\end{proof}

\begin{lemma}[using \axiomref{loc}, \axiomref{sqc}]%
  \label{parametrized-line-through-two-points-in-projective-space}
  Let $p \neq q \in \bP^n$ be given.
  Then there exists a map $f : \bP^1 \to \bP^n$
  such that $f([0 : 1]) = p$, $f([1 : 0]) = q$.
\end{lemma}

\begin{proof}
  What we want to prove is a proposition,
  so we can assume chosen $a, b \in \A^{n+1} \setminus \{0\}$
  with $p = [a]$, $q = [b]$.
  Then we set
  \[ f([x, y]) \colonequiv [xa + yb] \rlap{.}\]
  Let us check that $xa + yb \neq 0$.
  By \Cref{generalized-field-property},
  we have that $x$ or $y$ is invertible
  and both $a$ and $b$ have at least one invertible entry.
  If $xa = - yb$
  then it follows that $x$ and $y$ are both invertible
  and therefore $a$ and $b$ would be linearly equivalent,
  contradicting the assumption $p \neq q$.
  Of course $f$ is also well-defined
  with respect to linear equivalence in the pair $(x, y)$.
\end{proof}

\begin{lemma}[using \axiomref{loc}, \axiomref{sqc}]%
  \label{point-in-projective-space-apart-from-two-standard-points}
  Let $n \geq 1$.
  For every point $p \in \bP^n$,
  we have $p \neq [1 : 0 : 0 : \dots]$
  or $p \neq [0 : 1 : 0 : \dots]$.
\end{lemma}

\begin{proof}
  This is a special case of \Cref{projective-space-apartness-relation},
  but we can also give a very direct proof:
  Let $p = [a]$ with $a \in \A^{n+1} \setminus \{0\}$.
  By \Cref{generalized-field-property},
  there is an $i \in \{0, \dots, n\}$ with $a_i \neq 0$.
  If $i = 0$ then $p \neq [0 : 1 : 0 : \dots]$,
  if $i \geq 1$ then $p \neq [1 : 0 : 0 : \dots]$.
\end{proof}

\begin{theorem}[using \axiomref{loc}, \axiomref{sqc}]%
  \label{functions-on-projective-space-constant}
  All functions $\bP^n \to R$ are constant,
  that is,
  \[ H^0(\bP^n, R) \colonequiv (\bP^n \to R) = R \rlap{.} \]
\end{theorem}

\begin{proof}
  Let $f : \bP^n \to R$ be given.
  For any two distinct points $p \neq q : \bP^n$,
  we can apply \Cref{parametrized-line-through-two-points-in-projective-space}
  and (merely) find a map $\widetilde{f} : \bP^1 \to R$
  with $\widetilde{f}([0 : 1]) = f(p)$
  and $\widetilde{f}([1 : 0]) = f(q)$.
  Then we see $f(p) = f(q)$
  by \Cref{functions-on-projective-line-constant}.
  In particular, we have $f([1 : 0 : 0 : \dots]) = f([0 : 1 : 0 : \dots])$.
  And then, by \Cref{point-in-projective-space-apart-from-two-standard-points},
  we get $f(p) = f([1 : 0 : 0 : \dots])$ for every $p : \bP^n$.
\end{proof}

\begin{remark}
  Another proof of \Cref{functions-on-projective-space-constant}
  goes as follows:
  A function $f : \bP^n \to R$
  is by definition of $\bP^n$
  (\Cref{projective-space-hit})
  given by an $R^{\times}$-invariant function
  $g : \A^{n+1}\setminus\{0\} \to R$.
  But it is possible to show that the restriction function
  \[ (\A^{n+1} \to R) \xrightarrow{\sim} (\A^{n+1}\setminus\{0\} \to R) \]
  is bijective
  (as long as $n \geq 1$),
  so $g$ corresponds to a function $\widetilde{g} : \A^{n+1} \to R$
  which is constant on every subset of the form
  $\{\, rx \mid r : R^{\times} \,\}$
  for $x : \A^{n+1} \setminus \{0\}$.
  But then it is constant on the whole line
  $\{\, rx \mid r : R \,\}$,
  since the restriction function
  $(\A^1 \to R) \hookrightarrow (\A^1\setminus\{0\} \to R)$
  is injective.
  From this it follows that $f$ is constant with value $\widetilde{g}(0)$.

  A third possibility is to directly generalize the proof of
  \Cref{functions-on-projective-line-constant} to arbitrary $n$:
  The set $\bP^n$ is covered by the subsets $U_0, \dots, U_n$,
  so it is the colimit (in the category of sets) of a diagram
  of finite intersections of them,
  which are all affine schemes.
  The set of functions $\bP^n \to R$
  is thus the limit of a corresponding diagram of algebras.
  These algebras are most conveniently described as sub-algebras
  of the degree $0$ part of the graded algebra
  ${R[X_0, \dots, X_n]}_{X_0 \dots X_n}$,
  for example
  $(U_0 \to R) = R[\frac{X_1}{X_0}, \dots, \frac{X_n}{X_0}]$.
  Then the limit can be computed to be $R$.
\end{remark}

\subsection{Line Bundles}

We will construct Serre's twisting sheaves in this section,
starting with the ``minus first''.
The following works because of \Cref{lines-are-one-dimensional}.

We will also give some indication on which line bundles exist in general.

\begin{definition}%
  Let $X$ be a type.
  A \notion{line bundle} is a map $\mathcal L : X\to \Mod{R}$,
  such that
  \[ \prod_{x:X} \propTrunc{\mathcal L_x=_{\Mod{R}}R} \rlap{.}\]
  The \notion{trivial line bundle} on $X$ is the line bundle
  $X \to \Mod{R}, x \mapsto R$,
  and when we say that a line bundle $\mathcal{L}$ is trivial
  we mean that $\mathcal{L}$ is equal to the trivial line bundle,
  or equivalently $\propTrunc{\prod_{x:X} \mathcal L_x=_{\Mod{R}}R}$.
\end{definition}

\begin{definition}
  \begin{enumerate}[(a)]
  \item The \notion{tautological bundle} is the line bundle $\mathcal O_{\bP^n}(-1):\bP^n\to \Mod{R}$,
    given by
    \[ (L:\bP^n)\mapsto L\rlap{.}\]
  \item The \notion{dual}\index{$\mathcal L^\vee$} $\mathcal L^\vee$ of a line bundle $\mathcal L:\bP^n\to \Mod{R}$,
    is the line bundle given by
    \[ (x:\bP^n)\mapsto \Hom_{\Mod{R}}(\mathcal L_x,R)\rlap{.}\]
  \item The \notion{tensor product} of $R$-module bundles $\mathcal F\otimes \mathcal G$\index{$\mathcal F\otimes \mathcal G$}
    on a scheme $X$
    is given by pointwise taking the tensor product of $R$-modules.
  \item For $k:\Z$, the $k$-th \notion{Serre twisting sheaf} $\mathcal{O}_{\bP^n}(k)$ on $\bP^n$
    is given by taking the $-k$-th tensor power of $\mathcal O_{\bP^n}(-1)$ for negative $k$
    and the $k$-th tensor power of $\mathcal O_{\bP^n}(-1)^{\vee}$ otherwise.
  \end{enumerate}
\end{definition}

With this, we expect that many classical results can be reproduced.
It is however, far from clear, in what sense we can expect that every line bundle on $\bP^n$
is a Serre twisting sheaf.
The expectation is, that this is not the case,
since even on $\A^1$, we should not expect that all line bundles are trivial.
The background of this expectation is the external fact, that the relative Picard group of
the affine line over the spectrum of the base ring is not always trivial.

We will proceed by showing the claim about line bundles on $\A^1$,
which will require some preparation.

\begin{lemma}[using \axiomref{loc}, \axiomref{sqc}, \axiomref{Z-choice}]
  For every open subset $U : \A^1 \to \Prop$ of $\A^1$
  we have not not:
  either $U = \emptyset$
  or $U = D((X - a_1)\dots(X - a_n)) = \A^1 \setminus \{ a_1, \dots, a_n \}$
  for pairwise distinct numbers $a_1, \dots, a_n : R$.
\end{lemma}

\begin{proof}
  For $U = D(f)$,
  this follows from \Cref{polynomials-notnot-decompose}
  because $D(\alpha \cdot {(X - a_1)}^{e_1} \dots {(X - a_n)}^{e_n})
  = D((X - a_1) \dots (X - a_n))$.
  In general,
  we have $U = D(f_1) \cup \dots \cup D(f_n)$
  by \Cref{qc-open-affine-open},
  so we do not not get
  (that $U = \emptyset$ or)
  a list of elements $a_1, \dots, a_n : R$
  such that $U = \A^1 \setminus \{ a_1, \dots, a_n \}$.
  Then we can not not get rid of any duplicates in the list.
\end{proof}

\begin{lemma}[using \axiomref{loc}, \axiomref{sqc}, \axiomref{Z-choice}]%
  \label{decompose-invertible-function-on-intersection-in-A1}
  Let $U, V : \A^1 \to \Prop$ be two open subsets
  and let $f : U \cap V \to R^\times$ be a function.
  Then there do not not exist functions
  $g : U \to R^\times$ and
  $h : V \to R^\times$
  such that $f(x) = g(x)h(x)$ for all $x : U \cap V$.
\end{lemma}

\begin{proof}
  By \Cref{polynomials-notnot-decompose},
  we can assume
  \begin{align*}
    U \cup V &= D((X - a_1) \dots (X - a_k)) \rlap{,}\\
    U &= D((X - a_1) \dots (X - a_k) (X - b_1) \dots (X - b_l)) \rlap{,}\\
    V &= D((X - a_1) \dots (X - a_k) (X - c_1) \dots (X - c_m)) \rlap{,}\\
    U \cap V &= D((X - a_1) \dots (X - a_k) (X - b_1) \dots (X - b_l) (X - c_1) \dots (X - c_m))
    \rlap{,}
  \end{align*}
  where all linear factors are distinct.
  Then every function $f : U \cap V \to R^\times$ can
  by (\axiomref{sqc}), \Cref{polynomials-notnot-decompose}
  and comparing linear factors
  not not be written in the form
  \[ f = \alpha \cdot
     {(X - a_1)}^{e_1} \dots {(X - a_k)}^{e_k}
     {(X - b_1)}^{e'_1} \dots {(X - b_l)}^{e'_l}
     {(X - c_1)}^{e''_1} \dots {(X - c_m)}^{e''_m}
  \]
  with $\alpha : R^\times$, $e_i, e'_i, e''_i : \Z$.
  Other linear factors can not appear,
  since they do not represent invertible functions on $U \cap V$.
  Now we can write $f = gh$ as desired,
  for example with
  \begin{align*}
    g &= \alpha \cdot
    {(X - a_1)}^{e_1} \dots {(X - a_k)}^{e_k}
    {(X - b_1)}^{e'_1} \dots {(X - b_l)}^{e'_l} \rlap{,}\\
    h &=
    {(X - c_1)}^{e''_1} \dots {(X - c_m)}^{e''_m} \rlap{.}
  \end{align*}
\end{proof}

\begin{theorem}[using \axiomref{loc}, \axiomref{sqc}, \axiomref{Z-choice}]%
  \label{Gm-torsors-on-A1}
  Every $R^\times$-torsor on $\A^1$ (\Cref{torsor})
  does not not have a global section.
\end{theorem}

\begin{proof}
  Let $T$ be an $R^\times$-torsor on $\A^1$,
  that is,
  for every $x : \A^1$,
  $T_x$ is a set with a free and transitive $R^\times$ action
  and $\propTrunc{T_x}$.
  By (\axiomref{Z-choice}),
  we get a cover of $\A^1$ by open subsets $\A^1 = \bigcup_{i = 1}^n U_i$
  and local sections $s_i : (x : U_i) \to T_x$ of the bundle $T$.
  From this we can not not construct a global section by induction on $n$:
  Given any two local sections $s_i, s_j$ defined on $U_i, U_j$,
  let $f : U_i \cap U_j \to R^\times$ be the unique function with
  $f(x)s_i(x) = s_j(x)$ for all $x : U_i \cap U_j$.
  Then by \Cref{decompose-invertible-function-on-intersection-in-A1},
  we not not find $g : U_i \to R^\times$, $h : U_j \to R^\times$
  such that the sections
  $x \mapsto g(x)s_i(x)$ and $x \mapsto {h(x)}^{-1}s_j(x)$,
  defined on $U_i$ respectively $U_j$,
  agree on $U_i \cap U_j$.
  This yields a section $\widetilde{s} : (x : U_i \cup U_j) \to T_x$
  by \Cref{kraus-glueing}
  and we can replace $U$ and $V$ by $U \cup V$ in the cover.
  Finally, when we get to $n = 1$,
  we have $U_1 = \A^1$
  and the global section $s_1 : (x : X) \to T_x$.
\end{proof}

\begin{corollary}[using \axiomref{loc}, \axiomref{sqc}, \axiomref{Z-choice}]
  \label{line-bundle-A1-notnot-trivial}
  Every line bundle on $\A^1$ is not not trivial.
\end{corollary}

\begin{proof}
  Given a line bundle $\mathcal{L}$,
  we can construct an $R^\times$ torsor
  \[ x \mapsto \mathcal{L}_x \setminus \{0\} \rlap{.} \]
  Note that there is a well-defined $R^\times$ action on $M \setminus \{0\}$
  for every $R$ module $M$,
  and the action on $\mathcal{L}_x \setminus \{0\}$ is free and transitive
  and we have $\propTrunc{\mathcal{L}_x \setminus \{0\}}$
  since we merely have $\mathcal{L}_x = R$ as $R$ modules.
  By \Cref{Gm-torsors-on-A1},
  there not not is a global section of this torsor,
  so we have a section $s : (x : \A^1) \to \mathcal{L}_x$
  with $s(x) \neq 0$ for all $x : \A^1$.
  But this means that the line bundle $\mathcal{L}$ is trivial,
  since we can build an identification $\mathcal{L}_x = R$
  by sending $s(x)$ to $1$.
\end{proof}

We now transfer this result to line bundles on $\bP^1$.

\begin{lemma}[using \axiomref{loc}, \axiomref{sqc}]
  \label{invertible-laurent-polynomials}
  Every invertible element of the ring of Laurent polynomials $R[X]_X$
  is not not of the form $\alpha X^n$
  for some $\alpha : R^{\times}$ and $n : \Z$.
\end{lemma}

\begin{proof}
  Every element $f : R[X]_X$ is of the form
  $f = \sum_{i = m}^{m + n} a_i X^i$
  for some $m : \Z$, $n \geq 0$ and $a_i : R$.
  Every $a_i$ is not not either $0$ or invertible.
  Thus we can assume that either $f = 0$ or
  both $a_m$ and $a_{m+n}$ are invertible.
  If $f$ is invertible, we can exclude $f = 0$,
  and it remains to show that $n = 0$.
  Applying the same reasoning to $g$,
  where $fg = 1$,
  we see that $n > 0$ is indeed impossible.
\end{proof}

\begin{theorem}[using \axiomref{loc}, \axiomref{sqc}, \axiomref{Z-choice}]
  For every line bundle $\mathcal{L}$ on $\bP^1$,
  there not not exists a $k : \Z$
  such that $\mathcal{L} = \mathcal{O}_{\bP^1}(k)$.
\end{theorem}

\begin{proof}
  Let $\mathcal{L}:\bP^1\to \Mod{R}$ be a line bundle on $\bP^1$.
  By pushout recursion, $\mathcal{L}$ is given by two line bundles $\mathcal L_0,\mathcal L_1:\A^1\to\Mod{R}$
  and a glueing function $g:(x:\A^1\setminus\{0\})\to \mathcal L_0(x)=\mathcal L_1(\frac{1}{x})$.
  Since we are proving a double negation, we can assume identifications $p_0:(x:\A^1)\to \mathcal L_0=R^1$
  and $p_1:(x:\A^1)\to \mathcal L_1=R^1$ by \Cref{line-bundle-A1-notnot-trivial}.

  Now we can define $g':(x:\A^1\setminus\{0\})\to R^1=R^1$ by $g'(x)\colonequiv p_0^{-1}(x)\cdot g(x)\cdot p_1(\frac{1}{x})$.
  By synthetic quasi-coherence, equivalently, $g'$ is an invertible element of $R[X]_X$
  and therefore by \Cref{invertible-laurent-polynomials} given by $\alpha X^n$ for some $\alpha:R^\times$ and $n:\Z$.
  We can assume $\alpha=1$, since this just amounts to concatenating
  our final equality with the automorphism of line bundles given by $\alpha^{-1}$ at all points.

  By explicit calculation, the tautological bundle $\mathcal O_{\bP^1}(-1)$ on $\bP^1$ is given by glueing trivial line bundles along
  a glueing function $g_{-1}:(x:\A^1\setminus\{0\})\to R^1=R^1$ with $g_{-1}(x)\colonequiv \lambda\mapsto x\cdot \lambda$.
  Note that an arbitrary choice of sign is involved, made by choosing the direction of the glueing function.
  Sticking with the same choice, calculation shows $g_1(x)\colonequiv \lambda\mapsto \frac{1}{x}\cdot \lambda$
  is a glueing function for the dual of the tautological bundle $\mathcal O_{\bP^1}(1)$
  and the tensor product of line bundles corresponds to multiplication.
\end{proof}

\section{Bundles and cohomology}
In non-synthetic algebraic geometry,
the structure sheaf~$\mathcal{O}_X$ is part of the data constituting a scheme~$X$.
In our internal setting,
a scheme is just a type satisfying a property.
When we want to consider the structure sheaf as an object in its own right,
we can represent it by the trivial bundle
that assigns to every point $x : X$ the set $R$.
Indeed, for an affine scheme $X = \Spec A$,
taking the sections of this bundle over a basic open $D(f) \subseteq X$
\[ \left(\prod_{x : D(f)} R\right) = (D(f) \to R) = A[f^{-1}] \]
yields the localizations of the ring $A$
expected from the structure sheaf $\mathcal{O}_X$.
More generally,
instead of sheaves of abelian groups, $\mathcal{O}_X$-modules, etc.,
we will consider bundles of abelian groups, $R$-modules, etc.,
in the form of maps from $X$ to the respective type of algebraic structures.

\subsection{Quasi-coherent bundles}

Sometimes we want to ``apply'' a bundle to a subtype,
like sheaves can be evaluated on open subspaces
and introduce the common notation ``$M(U)$'' for that below.
It is, however, not justified to expect, that this application
and the corresponding theory of ``sheaves'' is ``the same'' as the external one,
since the definition below uses the internal hom ``$\prod$''
-- where the corresponding external construction, would be the set of continuous sections of a bundle.

\begin{definition}
  \index{$M(U)$}
  \label{application-of-bundle-to-subtype}
  Let $X$ be a type and $M:X\to \Mod{R}$ a dependent module.
  Let $U\subseteq X$ be any subtype.
  \begin{enumerate}[(a)]
  \item We write:
    \[
      M(U)\colonequiv \prod_{x:U}M_x
      \rlap{.}
    \]
  \item With pointwise structure, $U\to R$ is an $R$-algebra
    and $M(U)$ is a $(U\to R)$-module.
  \end{enumerate}
\end{definition}

Somewhat surprisingly, localization of modules $M(U)$
can be done pointwise:

\begin{lemma}[using \axiomref{loc}, \axiomref{sqc}, \axiomref{Z-choice}]%
  \label{module-bundle-localization-pointwise}
  Let $X$ be a scheme and $M:X\to \Mod{R}$ a dependent module.
  For any $f:X\to R$, there is an equality
  \[
    M(X)_f=\prod_{x:X}(M_x)_{f(x)}
  \]
  of $(X\to R)$-modules.
\end{lemma}

\begin{proof}
First we construct a map, by realizing that the following is well-defined:
\[
  \frac{m}{f^k}\mapsto\left(x\mapsto \frac{m(x)}{f(x)^k}\right)
\]
So let $\frac{m}{f^k}=\frac{m'}{f^{k'}}$,
i.e. let there be an $l:\N$ such that $f^l(mf^{k'}-m'f^k)=0$.
But then we can choose the same $l:\N$ for each $x:X$
and apply the equation to each $x:X$.

We will now show, that the map we defined is an embedding.
So let $g,h:M(X)_f$ such that $p:\prod_{x:X}g(x)=_{(M_x)_{f(x)}}h(x)$.
Let $m_g,m_h:\prod_{x:X} M_x$ and $k_g,k_h:\N$ such that
\[
  g=\frac{m_g}{f^{k_g}} \quad\text{and}\quad h=\frac{m_h}{f^{k_h}}
  \rlap{.}
\]
From $p$ we know $\prod_{x:X}\exists_{k_x:\N}f(x)^{k_x}(m_g(x)f(x)^{k_h}-m_h(x)f(x)^{k_g})=0$.
By \Cref{strengthened-boundedness},
we find one $k : \N$ with
\[
  \prod_{x:X}f(x)^{k}(m_g(x)f(x)^{k_h}-m_h(x)f(x)^{k_g})=0
\]
--- which shows $g=h$.

It remains to show that the map is surjective.
So let $\varphi:\prod_{x:X}(M_x)_{f(x)}$ and
note that
\[
  \prod_{x:X}
  \exists_{k_x:\N,m_x:M_x}.
  \varphi(x)=\frac{m_x}{f(x)^{k_x}}
  \rlap{.}
\]
By \Cref{strengthened-boundedness} and \Cref{zariski-choice-scheme},
we get $k:\N$, an affine open cover $U_1,\dots,U_n$ of $X$ and $m_i:(x : U_i)\to M_x$
such that for each $i$ and $x:U_i$ we have
\[
  \varphi(x)=\frac{m_i(x)}{f(x)^{k}}
  \rlap{.}
\]
The problem is now to construct a global $m:(x:X)\to M_x$ from the $m_i$.
We have
\[
    \prod_{x:U_{ij}}\frac{m_i(x)}{f(x)^k}=\varphi(x)=\frac{m_j(x)}{f(x)^k}
\]
meaning there is pointwise an exponent $t_x:\N$,
such that $f(x)^{t_x}m_i(x)=f(x)^{t_x}m_j(x)$.
By \Cref{strengthened-boundedness},
we can find a single $t:\N$ with this property and define
\[
  \tilde{m}_i(x) \colonequiv f(x)^t m_i(x)
  \rlap{.}
\]
Then we have $\tilde{m}_i(x)=\tilde{m}_j(x)$ on all intersections $U_{ij}$,
which is what we need to get a global $m:(x:X)\to M_x$ from \Cref{kraus-glueing}.
Since $\varphi(x)=\frac{f(x)^t m_i(x)}{f(x)^{t+k}}=\frac{\tilde{m}_i(x)}{f(x)^{t+k}}$
for all $i$ and $x : U_i$,
we have found a preimage of $\varphi$ in $M(X)_f$.
\end{proof}

We will need the following algebraic observation:

\begin{remark}%
  \label{localization-to-module-if-non-zero}
  Let $M$ be an $R$-module and $A$ a finitely presented $R$-algebra,
  then there is an $R$-linear map
  \[
    M\otimes A\to M^{\Spec A}
  \]
  induced by mapping $m\otimes f$ to $x\mapsto x(f)\cdot m$.
  In particular, for any $f:R$, there is a
  \[
    M_f\to M^{D(f)}
    \rlap{.}
  \]
  The map $M\otimes A\to M^{\Spec A}$ is natural in $M$.
\end{remark}

\begin{lemma}[using \axiomref{loc}, \axiomref{sqc}, \axiomref{Z-choice}]%
  \label{localization-to-restriction}                    
  Let $X$ be a scheme, $M:X\to\Mod{R}$, $U\subseteq X$ open and $f:A$.
  Then there is an $R$-linear map
  \[
    M(U)_f \to M(D(f)) 
    \rlap{.}
  \]
\end{lemma}

\begin{proof}
  Combining \Cref{module-bundle-localization-pointwise}
  and pointwise application of \Cref{localization-to-module-if-non-zero} we get
  \[
    M(U)_f=\left(\prod_{x:U}(M_x)_{f(x)}\right)\to \left(\prod_{x:U}(M_x)^{D(f(x))}\right)
    =\left(\prod_{x:D(f)}M_x\right)
    =M(D(f))
  \]
\end{proof}

A characterization of quasi coherent sheaves in the little Zariski-topos was found with \cite{ingo-thesis}[Theorem 8.3].
This characterization is similar to our following definition of weak quasi-coherence,
which will provide us with an abelian subcategory of the $R$-module bundles over a scheme,
where we can show that higher cohomology vanishes if the scheme is affine.

\begin{definition}
  \label{weakly-quasi-coherent-module}
  An $R$-module $M$ is \notion{weakly quasi-coherent},
  if for all $f:R$, the canonical homomorphism
  \[
    M_f\to M^{D(f)}
  \]
  from \Cref{localization-to-module-if-non-zero} is an equivalence.
  We denote the type of weakly quasi-coherent $R$-modules
  with $\Mod{R}_{wqc}$\index{$\Mod{R}_{wqc}$}.
\end{definition}

\begin{lemma}
  \label{kernel-wqc}
  For any $R$-linear map $f:M\to N$ of weakly quasi-coherent modules $M$ and $N$,
  the kernel of $f$ is weakly quasi-coherent.
\end{lemma}

\begin{proof}
  Let $K\to M$ be the kernel of $f$.
  For any $f:R$, the map $K^{D(f)}\to M^{D(f)}$ is the kernel of $M^{D(f)}\to N^{D(f)}$.
  The latter map is equal to $M_f\to N_f$ by weak quasi-coherence of $M$ and $N$
  and $K_f\to M_f$ is the kernel of $M_f\to N_f$.
  Let the vertical maps in
  \begin{center}
    \begin{tikzcd}
      K_f\ar[r]\ar[d] & M_f\ar[r]\ar[d,"\simeq"] & N_f\ar[d,"\simeq"] \\
      K^{D(f)}\ar[r] & M^{D(f)}\ar[r] & N^{D(f)}
    \end{tikzcd}
  \end{center}
  be the canonical maps from \Cref{localization-to-module-if-non-zero}.
  The squares commute because of the naturality of the vertical maps.
  Then the map $K_f\to K^{D(f)}$ is an isomorphism,
  because by commutativity, it is equal to the induced map between the kernels $K_f$ and $K^{D(f)}$,
  which has to be an isomorphism, since it is induced by an isomorphism of diagrams.
\end{proof}

\begin{definition}%
  \label{weakly-quasi-coherent-bundle}
  Let $X$ be a scheme.
  A weakly quasi-coherent bundle on $X$, is a map $M:X\to \Mod{R}_{wqc}$.
\end{definition}

An immediate consequence is, that
weakly quasi coherent dependent modules have
the property that ``restricting is the same as localizing'':

\begin{lemma}[using \axiomref{loc}, \axiomref{sqc}, \axiomref{Z-choice}]
  \label{weakly-quasi-coherent-open-localization}
  Let $X$ be a scheme and $M:X\to \Mod{R}$ weakly quasi-coherent,
  then for all open $U\subseteq X$ and $f:U\to R$
  the canonical morphism
  \[
    M(U)_f\to M(D(f))
  \]
  is an equivalence.
\end{lemma}

\begin{proof}
  By construction of the canonical map from \Cref{localization-to-restriction}.
\end{proof}

Let us look at an example.

\begin{proposition}%
  \label{fp-algebra-bundle-is-quasi-coherent}
  Let $X$ be a scheme and $C:X\to \Alg{R}_{fp}$.
  Then $C$, as a bundle of $R$-modules, is weakly quasi coherent.
\end{proposition}

\begin{proof}
  Then for any $f:R$ and $x:X$, using \Cref{algebra-valued-functions-on-affine}, we have
  \[
    (C_x)_f=C_x\otimes_R R_f=(\Spec R_f \to C_x)=(D(f)\to C_x)={C_x}^{D(f)}
    \rlap{.}
  \]
\end{proof}

For examples of non weakly quasicoherent modules,
see \Cref{non-wqc-module-family}
and \Cref{RN-non-wqc}.

\begin{lemma}[using \axiomref{loc}, \axiomref{sqc}, \axiomref{Z-choice}]%
  \label{weakly-quasi-coherent-pi}
  Let $X$ be an affine scheme and $M_x$ a weakly quasi-coherent $R$-module for any $x:X$,
  then
  \[
    \prod_{x:X}M_x
  \]
  is a weakly quasi-coherent $R$-module.
\end{lemma}

\begin{proof}
  We need to show:
  \[
    \left(\prod_{x:X}M_x\right)_f=\left(\prod_{x:X}M_x\right)^{D(f)}
  \]
  for all $f:R$.
  By weak \Cref{module-bundle-localization-pointwise}, quasi-coherence
  and \Cref{weakly-quasi-coherent-open-localization}
  we know:
  \[
    \left(\prod_{x:X}M_x\right)_f
    =\prod_{x:X}\left(M_x\right)_{f(x)}
    =\prod_{x:X}\left(M_x\right)^{D(f)}
    =\left(\prod_{x:X}M_x\right)^{D(f)}
    \rlap{.}
  \]
\end{proof}

Quasi-coherent dependent modules turn out to have very good properties,
which are to be expected from what is known about their external counterparts.
We will show below, that quasi coherence is preserved by the following constructions:

\begin{definition}
  \label{pullback-push-forward}
  Let $X,Y$ be types and $f:X\to Y$ be a map.
  \begin{enumerate}[(a)]
  \item \index{$f^*M$} For any dependent module $N:Y\to\Mod{R}$,
    the \notion{pullback} or \notion{inverse image} is the dependent module
    \[
      f^*N\colonequiv (x:X) \mapsto M_{f(x)}\rlap{.}
    \]
  \item \index{$f_*M$} For any dependent module $M:X\to\Mod{R}$,
    the \notion{push-forward} or \notion{direct image} is the dependent module
    \[
      f_*M\colonequiv (y:Y) \mapsto \prod_{x:\fib_f(y)}M_{\pi_1(x)}\rlap{.}
    \]
  \end{enumerate}
\end{definition}

\begin{theorem}[using \axiomref{loc}, \axiomref{sqc}, \axiomref{Z-choice}]%
  \label{pullback-push-forward-qcoh}
  Let $X,Y$ be schemes and $f:X\to Y$ be a map.
  \begin{enumerate}[(a)]
  \item For any weakly quasi-coherent dependent module $N:Y\to\Mod{R}$,
    the inverse image $f^*N$ is weakly quasi-coherent.
  \item For any weakly quasi-coherent dependent module $M:X\to\Mod{R}$,
    the direct image $f_*M$ is weakly quasi-coherent.
  \end{enumerate}
\end{theorem}

\begin{proof}
  \begin{enumerate}[(a)]
  \item There is nothing to do, when we use the pointwise definition of weak quasi-coherence. 
  \item We need to show, that
    \[
      \prod_{x:\fib_f(y)}M_{\pi_1(x)}
    \]
    is a weakly quasi-coherent $R$-module.
    By \Cref{fiber-product-scheme},
    the type $\fib_f(y)$ is a scheme.
    So by \Cref{weakly-quasi-coherent-pi},
    the module in question is weakly quasi-coherent.
  \end{enumerate}
\end{proof}

With a non-cyclic forward reference to a cohomological result,
there is a short proof of the following:

\begin{proposition}[using \axiomref{loc}, \axiomref{sqc}, \axiomref{Z-choice}]%
  Let $f:M\to N$ be an $R$-linear map of weakly quasi-coherent $R$-modules $M$ and $N$,
  then the cokernel $N/M$ is weakly quasi-coherent.
\end{proposition}

\begin{proof}
  We will first show, that for an $R$-linear embedding $m:M\to N$
  of weakly quasi-coherent $R$-modules $M$ and $N$,
  the cokernel $N/M$ is weakly quasi-coherent.
  We need to show:
  \[
    (N/M)_f=(N/M)^{D(f)}.
  \]
  By algebra: $(N/M)_f=N_f/M_f$.
  This means we are done, if $(N/M)^{D(f)}=N^{D(f)}/{M^{D(f)}}$.
  To see this holds, let us consider $0\to M\to N\to N/M\to 0$ as a short exact sequence of dependent modules,
  over the subtype of the point $D(f)\subseteq 1=\Spec R$.
  Then, taking global sections, by \Cref{cohomology-les},
  we have an exact sequence
  \[
    0\to M^{D(f)}\to N^{D(f)}\to (N/M)^{D(f)}\to H^1(D(f),M)
  \]
  -- but $D(f)=\Spec R_f$ is affine,
  so the last term is 0 by \Cref{H1-wqc-module-affine-trivial}
  and $(N/M)^{D(f)}$ is the cokernel $N^{D(f)}/M^{D(f)}$.

  Now we will show the statement for a general $R$-linear map $f:M\to N$.
  By algebra, the cokernel of $f$ is the same as the cokernel of the induced map
  $M/K\to N$, where $K$ is the kernel of $f$.
  By \Cref{kernel-wqc}, $K$ is weakly quasi-coherent, so by the proof above,
  $M/K$ is weakly quasi-coherent.
  $M/K\to N$ is an embedding, so again by the proof above, its cokernel is weakly quasi-coherent.
\end{proof}

\subsection{Finitely presented bundles}

We now investigate the relationship between bundles of $R$-modules on $X = \Spec A$
and $A$-modules.

\begin{proposition}
  Let $A$ be a finitely presented $R$-algebra.
  There is an adjunction
  \[ \begin{tikzcd}[row sep=tiny]
    M \ar[r, mapsto] & {(M \otimes x)}_{x : \Spec A} \\
    \Mod{A} \ar[r, shift left=2] \ar[r, phantom, "\rotatebox{90}{$\vdash$}"] &
    \Mod{R}^{\Spec A} \ar[l, shift left=2] \\
    \prod_{x : \Spec A} N_x & N \ar[l, mapsto]
  \end{tikzcd} \]
  between the category of $A$-modules
  and the category of bundles of $R$-modules on $\Spec A$.
\end{proposition}

For an $A$-module $M$,
the unit of the adjunction is:
\begin{align*}
  \eta_M : M &\to \prod_{x : \Spec A} (M \otimes x) \\
  m &\mapsto (m \otimes 1)_{x : \Spec A}
\end{align*}

\begin{example}[using \axiomref{sqc}, \axiomref{loc}]
  It is not the case that
  for every finitely presented $R$-algebra $A$
  and every $A$-module $M$
  the map $\eta_M$ is injective.
\end{example}

\begin{proof}
  \cite{topology-draft}.
\end{proof}

\begin{theorem}%
  \label{fp-module}
  Let $X=\Spec(A)$ be affine and
  let a bundle of finitely presented $R$-modules $M : X\to \fpMod{R}$ be given.
  Then the $A$-module
  \[ \tilde{M}\coloneqq\prod_{x:X}M_x \]
  is finitely presented and for any $x:X$ the $R$-module $\tilde{M}\otimes_A R$ is $M_x$.
  Under this correspondence, localizing $\tilde{M}$ at $f:A$ corresponds to restricting $M$ to $D(f)$.
\end{theorem}

\subsection{Cohomology on affine schemes}

\begin{definition}%
  \label{torsor}
  Let $X$ be a type and $A:X\to \AbGroup$ a map to the type of abelian groups.
  For $x:X$ let $T_x$ be a set with an $A_x$ action.
  \begin{enumerate}[(a)]
  \item $T$ is an \notion{$A$-pseudotorsor}, if the action is free and transitive for all $x:X$.
  \item $T$ is an \notion{$A$-torsor}, if it is an $A$-pseudotorsor and
    \[ \prod_{x:X} \propTrunc{ T_x } \rlap{.}\]
  \item We write $\Tors{A}(X)$ for the type of $A$-torsors on $X$.
  \end{enumerate}
\end{definition}

Torsors on a point are a concrete implementation of first deloopings:

\begin{definition}
  \label{delooping}
  Let $n:\N$.
  A $n$-th \notion{delooping}\index{$K(A,n)$} of an abelian group $A$,
  is a pointed, $(n-1)$-connected, $n$-truncated type $K(A,n)$,
  such that $\Omega^nK(A,n)=_{\AbGroup}A$.
\end{definition}

For any abelian group and any $n$, a delooping $K(A,n)$ exists by \cite{licata-finster}.
Deloopings can be used to represent cohomology groups by mapping spaces.
This is usually done in homotopy type theory to study higher inductive types, such as spheres and CW-complexes,
but the same approach works for internally representing sheaf cohomology,
which is the intent of the following definition:

\begin{definition}
  \label{cohomology}
  Let $X$ be a type and $\mathcal F:X\to\AbGroup$ a dependent abelian group.
  The $n$-th cohomology group of $X$ with coefficients in $\mathcal F$ is
  \[
    H^n(X,\mathcal F)\colonequiv \left\propTrunc{\prod_{x:X}K(\mathcal F,n)\right}_0\rlap{.}
  \]
\end{definition}

\begin{theorem}%
  \label{cohomology-les}
  Let $\mathcal F,\mathcal G,\mathcal H:X\to \AbGroup$ be such that for all $x:X$,
  \[
    0\to \mathcal F_x\to\mathcal G_x\to\mathcal H_x\to 0
  \]
  is an exact sequence of abelian groups. Then there is a long exact sequence:
  \begin{center}
    \begin{tikzcd}
      & \dots\ar[r] & H^{n-1}(X,\mathcal H)\ar[dll] \\
      H^n(X,\mathcal F)\ar[r] & H^n(X,\mathcal G)\ar[r] & H^n(X,\mathcal H)\ar[dll] \\
      H^{n+1}(X,\mathcal F)\ar[r] & \dots &
    \end{tikzcd}
  \end{center}
\end{theorem}

\begin{proof}
  By applying the long exact homotopy fiber sequence.
\end{proof}

The following is an explicit formulation of the fact, that the Čech-Complex for an
$\mathcal{O}_X$-module sheaf on $X=\Spec(A)$ given by an $A$-module $M$ is exact in degree 1.
\begin{lemma}%
  \label{H1-algebra}
  Let $M$ be a module over a commutative ring $A$, $F_1,\dots,F_l$ a coprime system on $A$
  and for $i,j\in\{1,\dots,l\}$, let $s_{ij} : F_i^{-1} F_j^{-1} M$ such that:
  \[ s_{jk}-s_{ik}+s_{ij}=0 \rlap{.}\]
  Then there are $u_i:F_i^{-1}M$ such that $s_{ij}=u_j - u_i$.
\end{lemma}

\begin{proof}
  Let $s_{ij}=\frac{m_{ij}}{f_i f_j}$ with $m_{ij}:M$, $f_i:F_i$ and $f_j:F_j$ such that:
  \[ f_i\cdot m_{jk}-f_j\cdot m_{ik}+f_k\cdot m_{ij}=0 \rlap{.}\]
  Let $r_i$ such that $\sum r_i f_i =1$.
  Then for
  \[ u_i \coloneqq -\sum_{k=1}^l\frac{r_k}{f_i}m_{ik} \]
  we have:
  \begin{align*}
      u_j-u_i &= -\sum_{k=1}^l\frac{r_k}{f_j}m_{jk} + \sum_{k=1}^l\frac{r_k}{f_i}m_{ik} \\
              &= -\sum_{k=1}^l\frac{r_k}{f_j f_i}f_i m_{jk} + \sum_{k=1}^l\frac{r_k}{f_i f_j} f_j m_{ik} \\
              &= \sum_{k=1}^l\frac{r_k}{f_j f_i}(-f_i m_{jk} + f_j m_{ik}) \\
              &= \sum_{k=1}^l\frac{r_k}{f_j f_i}f_k m_{ij} \\
              &= \frac{m_{ij}}{f_i f_j}
  \end{align*}
  \ %
\end{proof}

\begin{theorem}[using \axiomref{loc}, \axiomref{sqc}, \axiomref{Z-choice}]%
  \label{H1-wqc-module-affine-trivial}
  For any affine scheme $X=\Spec(A)$ and coefficients $M: X\to \Mod{R}_{wqc}$, we have
  \[ H^1(X,M)=0 \rlap{.} \]
\end{theorem}

\begin{proof}
  We need to show, that any $M$-torsor $T$ on $X$ is merely equal to the trivial torsor $M$,
  or equivalently show the existence of a section of $T$.
  We have
  \[ \prod_{x:X}\propTrunc{ T_x }\]
  and therefore, by (\axiomref{Z-choice}),
  there merely are $f_1,\dots,f_l:A$,
  such that the $U_i\coloneqq \Spec(A_{f_i})$ cover $X$ and
  there are local sections
  \[ s_i:\prod_{x:U_i}T_x\]
  of $T$. Our goal is to construct a matching family from the $s_i$.
  On intersections, let $t_{ij}\coloneqq s_i-s_j$ be the difference, so $t_{ij}:(x : U_i\cap U_j) \to M_x$.
  By \Cref{weakly-quasi-coherent-open-localization} equivalently,
  we have $t_{ij}:M(U_{i}\cap U_j)_{f_i f_j}$.
  Since the $t_{ij}$ were defined as differences,
  the condition in \Cref{H1-algebra} is satisfied and we get
  $u_i:M(U_i)_{f_i}$, such that $t_{ij}=u_i-u_j$.
  So we merely have a matching family $\tilde{s}_i\coloneqq s_i-u_i$ and therefore, using Lemma \ref{kraus-glueing} merely a section of $T$.
\end{proof}

A similar result is provable for $H^2(X,M)$ using the same approach.
There is an extension of this result to general $n$ in work in progress \cite{chech-draft}.

\subsection{Čech-Cohomology}

In this section, let $X$ be a type, $U_1,\dots,U_n\subseteq X$ open subtypes that cover $X$
and $\mathcal F:X\to \AbGroup$ a dependent abelian group on $X$.
We start by repeating the classical definition of \v{C}hech-Cohomology groups for a given cover.

\begin{definition}%
  \label{chech-complex}
  \begin{enumerate}[(a)]
  \item \index{$\mathcal F(U)$} For open $U\subseteq X$, we use the notation from \Cref{application-of-bundle-to-subtype}:
    \[
      \mathcal F(U)\colonequiv \prod_{x:U}\mathcal F_x\rlap{.}
    \]
  \item For $s:\mathcal F(U)$ and open $V\subseteq U$ we use the notation $s\colonequiv s_{|V} \colonequiv (x:V)\mapsto s_x$.
  \item \index{$U_{i_1\dots i_l}$}For a selection of indices $i_1,...,i_l:\{1,\dots,n\}$, we use the notation
    \[
      U_{i_1\dots i_l}\colonequiv U_{i_1}\cap\dots\cap U_{i_l}\rlap{.}
    \]
  \item For a list of indices $i_1,\dots,i_l$, let $i_1,\dots,\hat{i_t},\dots,i_l$ be the same list with the $t$-th element removed.
  \item For $k:\Z$, the $k$-th \notion{Čech-boundary operator}\index{$\partial^k$} is the homomorphism
    \[
      \partial^k:\bigoplus_{i_0,\dots,i_k}\mathcal F(U_{i_0\dots i_k})\to \bigoplus_{i_0,\dots,i_{k+1}}\mathcal F(U_{i_0\dots i_{k+1}})
    \]
    given by $\partial^k(s)\colonequiv (l_0,\dots,l_{k+1}) \mapsto \sum_{j=0}^k (-1)^j s_{l_0,\dots,\hat{l_j},\dots,l_k|U_{l_0,\dots,l_{k+1}}}$.
  \item The $k$-th \notion{Čech-Cohomology group} for the cover $U_1,\dots,U_n$ with coefficients in $\mathcal F$ is
    \[
      \check{H}^k(\{U\},\mathcal F)\colonequiv \ker\partial^{k} / \im(\partial^{k-1})\rlap{.}
    \]
  \end{enumerate}
\end{definition}

It is possible to construct a torsor from a \v{C}ech cocycle:

\begin{lemma}%
  \label{deligne-construction}
  Let $A$ be an abelian group and $L$ a type with $\propTrunc{L}$.
  Let us call $c:(i,j:L)\to A$ a $L$-cocycle, if $c_{ij}+c_{jk}=c_{ik}$ for all $i,j,k:L$.
  Then there is a bijection:
  \[
    \left((T:\text{$A$-torsor})\times T^L\right) \to \text{$L$-cocycles}
    \rlap{.}
  \]
\end{lemma}

\begin{proof}
  Let us first check, that the left side is a set.
  Let $(T,u),(T',u'):(T:\text{$A$-torsor})\times T^L$,
  then $(T,u)=(T',u')$ is equivalent to $(e:T\cong T')\times ((i:L)\to e(u_i)=u'_i)$.
  But two maps $e$ with this property are equal,
  since a map between torsors is determined by the image of a single element and $L$ is inhabited.
  
  Assume now $(T,u):(T:\text{$A$-torsor})\times T^L$ to construct the map.
  Then $c_{ij}\colonequiv u_i-u_j$ defines an $L$-cocycle,
  because
  \[
    u_i-u_j + u_j-u_k = u_i-u_k
    \rlap{.}
  \]
  This defines an embedding: Assume $(T,u)$ and $(T',u')$ define the same $L$-cocycle,
  then $u_i-u_j=u'_i-u'_j$ for all $i,j:L$.
  We want to show a proposition, so we can assume there is $i:L$ and use that to get a map $e:T\to T'$
  that sends $u_i$ to $u'_i$.
  But then we also have
  \[
    e(u_j)=e(u_j-u_i+u_i)=e(u'_j-u'_i+u_i)=u'_j-u'_i+e(u_i)=u'_j-u'_i+u'_i=u'_j
  \]
  for all $j:L$, which means $(T,u)=(T',u')$.
    
  Now let $c$ be an $L$-cocycle.
  Following \cite{Deligne91}[Section 5.2], we can define a preimage-candidate:
  \[
    T_c\colonequiv \{u:A^L\mid u_i-u_j=c_{ij}\}
    \rlap{.}
  \]
  $A$ acts on $T_c$ pointwise, since $(a+u_i)-(a+u_j)=u_i-u_j=c_{ij}$ for all $a:A$.
  
  To show that $T_c$ is inhabited,
  we may assume $i_0:L$.
  Then we define $u_i\colonequiv -c_{i_0i}$ to get $u_i-u_j=-c_{i_0i}+c_{i_0j}=c_{ij}$.

  Now $c$ is of type $(A^L)^L=A^{L\times L}$, so we have an element of the left hand side.
  Applying the map constructed above yields a cocycle
  \[
    \tilde{c}_{ij}=(k\mapsto c_{ki})-(k\mapsto c_{kj})=(k\mapsto c_{ki}-c_{kj})=(k\mapsto c_{kj}+c_{ji}-c_{kj})=(k\mapsto c_{ji})
  \]
  -- so $(T_c,c)$ is a preimage of $c_{ij}$.
\end{proof}

\begin{definition}
  The cover $U_1,\dots,U_n$ is called \notion{r-acyclic} for $\mathcal F$,
  if we have the following triviality of higher (non Čech) cohomology groups:
  \[
    \forall l, r\geq l>0\ \forall i_0,\dots,i_{r-l}. H^l(U_{i_0,\dots,i_{r-l}},\mathcal F)=0\rlap{.}
  \]
\end{definition}

\begin{example}
  If $X$ is a scheme, $U_1,\dots,U_n$ a cover by affine open subtypes
  and $\mathcal F$ pointwise a weakly quasi coherent $R$-module,
  then $U_1,\dots,U_n$ is 1-acyclic for $\mathcal F$ by \Cref{H1-wqc-module-affine-trivial}.
\end{example}

\begin{theorem}[using \axiomref{Z-choice}]%
  If $U_1,\dots,U_n$ is a 1-acyclic cover for $\mathcal F$, then
  \[
    \check{H}^1(\{U\},\mathcal F)=H^1(X,\mathcal F)\rlap{.}
  \]
\end{theorem}

\begin{proof}
  Let $\pi$ be the projection map
  \[
    \pi :
    \left(
      \sum_{T:\Tors{\mathcal F}(X)}\prod_{i}\prod_{x:U_i}T_x
    \right)
    \to \Tors{\mathcal F}(X)\rlap{.}
  \]
  Let us abbreviate the left hand side with $T(\mathcal F,U)$.
  Since the cover is 1-acyclic, $\pi$ is surjective.
  With $L_x\colonequiv \sum_{i}U_i(x)$ and \Cref{deligne-construction} we get:
  \begin{align*}
    T(\mathcal F,U)&=\prod_{x:X}(T_x:\Tors{\mathcal F_x})\times T_x^{L_x} \\
                   &=\prod_{x:X}\text{$L_x$-cocycles}
                     \rlap{.}
  \end{align*}
  The latter is the type of \v{C}ech-1-cocycles (\Cref{chech-complex} (e))
  and in total the equality is given by the isomorphism
  \[
    (T,t) \mapsto (i,j\mapsto t_i - t_j) :
    T(\mathcal F,U)
    \to
    \ker(\partial^1)
    \subseteq
    \bigoplus_{i,j}\mathcal F(U_{ij})\rlap{.}
  \]

  Realizing, that $\im(\partial^0)$ corresponds to the subtype of $T(\mathcal F,U)$ of trivial torsors,
  we arrive at the following diagram:
  \begin{center}
    \begin{tikzcd}
      & \Tors{\mathcal F}(X)\ar[r,->>] & H^1(X,\mathcal F) \\
      \sum_{T:T(\mathcal F,U)}\propTrunc{\pi_1(T)=\mathcal F}\ar[r,hook] & T(\mathcal F,U)\ar[u,->>]\ar[d,equal] & \\
      \im{\partial^0}\ar[r,hook]\ar[u,equal] & \ker{\partial^1}\ar[r,->>] & \check{H}^1(\{U\},\mathcal F)
    \end{tikzcd}
  \end{center}
  The composed map $T(\mathcal F,U)\to H^1(X,\mathcal F)$ is a homomorphism
  and therefore by \Cref{surjective-abgroup-hom-is-cokernel} a cokernel.
  So the two cohomology groups are equal, since they are cokernels of the same diagram.
\end{proof}

It is possible to pass from torsors to gerbes,
which are the degree 2 analogue of torsors:

\begin{definition}
  \label{gerbe}
  Let $A:\AbGroup$ be an abelian group.
  An \notion{$A$-banded gerbe}, is a connected type $\mathcal G:\mathcal U$,
  together with, for all $y:\mathcal G$ an identification of groups $\Omega (\mathcal G,y)=A$.
\end{definition}

Analogous to the type of $A$-torsors, the type of $A$-banded gerbes is a second delooping of an abelian group $A$.
We can formulate a second degree version of \Cref{deligne-construction}:

\begin{theorem}
  \label{deligne-construction-gerbes}
  Let $A$ be an abelian group and $L$ a type with $\propTrunc{L}$.
  Let us call $c:(i,j,k : L)\to A$ a $L$-2-cocycle,
  if $c_{jkl}-c_{ikl}+c_{ijl}-c_{ijk}=0$ for all $i,j,k,l : L$.
  Then there is a bijection:
  \[
    \left((\mathcal G:\text{$A$-gerbe})\times (u:\mathcal G^{L})\times (i,j : L) \to u_i=u_j\right) \to \text{$L$-2-cocycle}
    \rlap{.}
  \]
\end{theorem}

This is provable, again, by translating Deligne's argument \cite{Deligne91}[Section 5.3].
Using this, the correspondence of Eilenberg-MacLane-Cohomology and \v{C}ech-Cohomology can be extended in the following way:

\begin{theorem}
  If $U_1,\dots,U_n$ is a 2-acyclic cover for $\mathcal F$, then
  \[
    \check{H}^2(\{U\},\mathcal F)=H^2(X,\mathcal F)\rlap{.}
  \]  
\end{theorem}

However, with this approach, we need versions of \Cref{kraus-glueing}, with increasing truncation level.
While this suggests, we can prove the correspondence for any cohomology group of \emph{external} degree $l$,
there is follow-up work in progress \cite{chech-draft},
which proves the correspondence for all \emph{internal} $l:\N$.
In the same draft, there is also a version of the vanishing result for all internal $l$.
This means that many of the usual, essential computations with \v{C}ech-Cohomology can be transferred to synthetic algebraic geometry.

\section{Type Theoretic justification of axioms}
\newcommand{\inc}{\mathsf{inc}}
\newcommand{\inl}{\mathsf{inl}}
\newcommand{\inr}{\mathsf{inr}}
\newcommand{\idd}{\mathsf{id}}
\newcommand{\II}{\mathbf{I}}
\newcommand{\nats}{\mathbb{N}}
\newcommand\norm[1]{\left\lVert #1 \right\rVert}

\newcommand{\Gm}{\mathsf{G_m}}
\newcommand{\ext}{\mathsf{ext}}
\newcommand{\patch}{\mathsf{patch}}
\newcommand{\cov}{\mathsf{cov}}
\newcommand{\isSheaf}{\mathsf{isSheaf}}
\newcommand{\isIso}{\mathsf{isIso}}
\newcommand{\Fib}{\mathsf{Fib}}

\newcommand{\Typp}{\mathsf{Type}}
\newcommand{\Elem}{\mathsf{Elem}}
\newcommand{\Cont}{\mathsf{Cont}}

\newcommand{\BB}{\square}
\newcommand{\CC}{\mathcal{C}}
\newcommand{\UU}{\mathcal{U}}
\newcommand{\WW}{\mathcal{W}}
\newcommand{\VV}{\mathcal{V}}

In this section, we present a model of the 3 axioms stated in \Cref{statement-of-axioms}.
This model is best described as an \emph{internal} model
of a presheaf model. The first part can then be described purely syntactically, starting from any model
of 4 other axioms that are valid in a suitable \emph{presheaf} model. We obtain then the sheaf model by defining
a family of open left exact modalities, and the new model is the model of types that are modal for all these modalities.
This method works both in a $1$-topos framework and for models of univalent type theory.

\subsection{Internal sheaf model}

\subsubsection{Axioms for the presheaf model}

We start from 4 axioms. The 3 first axioms can be seen as variation of our 3 axioms for synthetic algebraic geometric.

\begin{enumerate}[(1)]
\item $R$ is a ring,
\item for any f.p.\ $R$-algebra $A$, the canonical map $A\rightarrow R^{\Spec(A)}$ is an equivalence
\item for any f.p.\ $R$-algebra $A$, the set $\Spec(A)$ satisfies choice, which can be formulated as
  the fact that for any family of types $P(x)$ for $x:\Spec(A)$ there is a map
  $(\Pi_{x:\Spec(A)}\norm{P(x)})\rightarrow \norm{\Pi_{x:\Spec(A)}P(x)}$.
\item for any f.p.\ $R$-algebra $A$, the diagonal map $\nats\rightarrow\nats^{\Spec(A)}$ is an equivalence.
\end{enumerate}

As before, $\Spec(A)$ denotes the type of $R$-algebra maps from $A$ to $R$, and
if $r$ is in $R$, we write $D(r)$ for the proposition $\Spec(R_r)$.

Note that the first axiom does not require
$R$ to be local, and the third axiom states that $\Spec(A)$ satisfies \emph{choice} and not only Zariski local choice,
for any f.p. $R$-algebra $A$.

\subsubsection{Justification of the axioms for the presheaf model}

\newcommand{\FP}{\mathsf{FP}}

We justify briefly the second axiom (synthetic quasi-coherence). This justification will be done
in a $1$-topos setting, but exactly the same argument holds in the setting of presheaf models of
univalent type theory, since it only involves strict presheaves. A similar direct verification holds
for the other axioms.

We work with presheaves on the opposite of the category of finitely presented $k$-algebra. We write
$L,M,N,\dots$ for such objects, and $f,g,h,\dots$ for the morphisms. A presheaf $F$ on this category is given
by a collection of sets $F(L)$ with restriction maps $F(L)\rightarrow F(M),~u\mapsto f u$ for
$f:L\rightarrow M$ satisfying the usual uniformity conditions.

We first introduce the presheaf $\FP$ of {\em finite presentations}. This is internally the type
$$
\Sigma_{n:\nats}\Sigma_{m:\nats}R[X_1,\dots,X_n]^m
$$
which is interpreted by $\FP(L) = \Sigma_{n:\nats}\Sigma_{m:\nats}L[X_1,\dots,X_n]^m$.
If $\xi = (n,m,q_1,\dots,q_m)\in\FP(L)$ is such a presentation, we build a natural extension
$\iota:L\rightarrow L_{\xi} = L[X_1,\dots,X_n]/(q_1,\dots,q_m)$ where the system $q_1 = \dots = q_m = 0$
has a solution $s_{\xi}$. Furthermore, if we have another extension $f:L\rightarrow M$
and a solution $s\in M^n$ of this system in $M$, there exists a unique map $i(f,s):L_{\xi}\rightarrow M$
such that $i(f,s) s_{\xi} = s$ and $i(f,s)\circ \iota = f$.
Note that $i(\iota,s_{\xi}) = \id$.

\medskip

Internally, we have a map $A:\FP\rightarrow R\mathsf{-alg}(\UU_0)$, which to any presentation
$\xi = (n,m,q_1,\dots,q_m)$ associates the $R$-algebra $A(\xi) = L[X_1,\dots,X_n]/(q_1,\dots,q_m)$.
This corresponds externally to the presheaf on the category of elements of $\FP$ defined
by $A(L,\xi) = L_{\xi}$.

Internally, we have a map $\Spec(A):\FP\rightarrow \UU_0$, defined by $\Spec(A)(\xi) = Hom(A(\xi),R)$.
We can replace it by the isomorphic map which to $\xi = (n,m,q_1,\dots,q_m)$ associates the set
$S(\xi)$ of solutions of the system $q_1=\dots=q_m= 0$ in $R^n$.
Externally, this corresponds to the presheaf on the category of elements of $\FP$ so that
$\Spec(A)(L,n,m,q_1,\dots,q_m)$ is the set of solutions of the system $q_1=\dots=q_m=0$ in $L^n$.

\medskip

We now define externally two inverse maps $\varphi:A(\xi)\rightarrow R^{\Spec(A(\xi))}$ and
$\psi:R^{\Spec(A(\xi))}\rightarrow A(\xi)$.

\medskip

Notice first that $R^{\Spec(A)}(L,\xi)$, for $\xi = (n,m,q_1,\dots,q_m)$,
is the set of families of elements $l_{f,s}:M$ indexed by $f:L\rightarrow M$
and $s:M^n$ a solution of $fq_1 = \dots = fq_m=0$, satisfying the uniformity condition
$g(l_{f,s}) = l_{(g\circ f),gs}$ for $g:M\rightarrow N$.

\medskip

For $u$ in $A(L,\xi) = L_{\xi}$ we define $\varphi~u$ in $R^{\Spec(A)}(L,\xi)$ by
$$
(\varphi~u)_{f,s} = i(f,s)~u
$$
and for $l$ in $R^{\Spec(A)}(L,\xi)$ we define $\psi~l$ in $A(L,\xi) = L_{\xi}$ by
$$
\psi~ l = l_{\iota,s_{\xi}}
$$
These maps are natural, and one can check
$$
\psi~(\varphi~u) = (\varphi~u)_{\iota,s_{\xi}} = i(\iota,s_{\xi})~u = u
$$
and
$$
(\varphi~(\psi~l))_{f,s} = i(f,s)~(\psi~l) = i(f,s)~l_{\iota,s_{\xi}} = l_{(i(f,s)\circ \iota),(i(f,s)~s_{\xi})} = l_{f,s}
$$
which shows that $\varphi$ and $\xi$ are inverse natural transformations.

Furthermore, the map $\varphi$ is the external version of the canonical map $A(\xi)\rightarrow R^{\Spec(A(\xi))}$.
The fact that this map is an isomorphism is an (internally) equivalent statement of the second axiom.

\subsubsection{Sheaf model obtained by localisation from the presheaf model}

We define now a family of propositions. As before, if $A$ is a ring, we let $\Um(A)$ be the type of unimodular sequences
(\Cref{unimodular})
$f_1,\dots,f_n$ in $A$, i.e.\ such that $(1) = (f_1,\dots,f_n)$. To any element $\vec{r} = r_1,\dots,r_n$
in $\Um(R)$ we associate
the proposition $D(\vec{r}) = D(r_1)\vee\dots\vee D(r_n)$. If $\vec{r}$ is the empty sequence then
$D(\vec{r})$ is the proposition $1 =_R 0$. 

  Starting from any model of dependent type theory with univalence satisfying the 4 axioms above, we build a new
  model of univalent type theory by considering the types $T$ that are modal for all modalities defined by the propositions
  $D(\vec{r})$, i.e.\ such that all diagonal maps $T\rightarrow T^{D(\vec{r})}$ are equivalences.
  This new model is called the \emph{sheaf model}.

    This way of building a new sheaf model can be described purely syntactically, as in \cite{Quirin16}. In \cite{CRS21}, we extend
    this interpretation to cover inductive data types. In particular, we describe there the sheafification $\nats_S$ of the type
    of natural numbers with the unit map $\eta:\nats\rightarrow\nats_S$. 

    A similar description can be done starting with the $1$-presheaf model. In this case, we use for the propositional truncation of a
    presheaf $A$ the image of the canonical map $A\rightarrow 1$. We however get a model of type theory {\em without} universes when we
    consider modal types.

    \begin{proposition}\label{modal}
      The ring $R$ is modal. It follows that any f.p.\ $R$-algebra is modal.
    \end{proposition}

    \begin{proof}
      If $r_1,\dots,r_n$ is in $\Um(R)$, we build a patch function $R^{D(r_1,\dots,r_n)}\rightarrow R$.
      Any element $u:R^{D(r_1,\dots,r_n)}$ gives a compatible family of elements $u_i:R^{D(r_i)}$, hence
      a compatible family of elements in $R_{r_i}$ by quasi-coherence. But then it follows from local-global
      principle \cite{lombardi-quitte}, that we can patch this family to a unique element of $R$.
      
      If $A$ is a f.p.\ $R$-algebra, then $A$ is isomorphic to $R^{\Spec(A)}$ and hence is modal.
    \end{proof}

    \begin{proposition}
      In this new sheaf model, $\perp_S$ is $1 =_R 0$.
    \end{proposition}

    \begin{proof}
      The proposition $1=_R0$ is modal by the previous proposition.
      If $T$ is modal, all diagonal maps $T\rightarrow T^{D(\vec{r})}$ are equivalences. For the empty sequence $\vec{r}$
      we have that $D(\vec{r})$ is $\perp$, and the empty sequence is unimodular exactly when $1 =_R 0$. So $1=_R0$
      implies that $T$ and $T^{\perp}$ are equivalent, and so implies that $T$ is contractible. By extensionality,
      we get that $(1=_R0)\rightarrow T$ is contractible when $T$ is modal.
    \end{proof}
    
    \begin{lemma}\label{Um}
      For any f.p.\ $R$-algebra $A$, we have $\Um(R)^{\Spec(A)} = \Um(A)$.
    \end{lemma}

    \begin{proof}
      Note that the fact that $r_1,\dots,r_n$ is unimodular is expressed by
      $$\norm{\Sigma_{s_1,\dots,s_n:R}r_1s_1+\dots+r_ns_n = 1}$$
      and we can use these axioms 2 and 3 to get
      $$\norm{\Sigma_{s_1,\dots,s_n:R}r_1s_1+\dots+r_ns_n = 1}^{\Spec(A)} = \norm{\Sigma_{v_1,\dots,v_n:A}\Pi_{x:\Spec(A)}r_1v_1(x)+\dots+r_nv_n(x) = 1}$$
      The result follows then from this and axiom 4.
    \end{proof}      




    
    For an f.p.\ $R$-algebra $A$, we can define the type of presentations $Pr_{n,m}(A)$ as the type $A[X_1,\dots,X_n]^m$.
    Each element in $Pr_{n,m}(A)$ defines an
    f.p.\ $A$-algebra. Since $Pr_{n,m}(A)$ is a modal type since $A$ is f.p., the type of presentations $Pr_{n,m}(A)_S$ in the sheaf model
    defined for $n$ and $m$ in $\nats_S$ will be such that $Pr_{\eta p,\eta q}(A)_S = Pr_{p,q}(A)$ \cite{CRS21}.
    
    \begin{lemma}\label{propsheaf}
      If $P$ is a proposition, then the sheafification of $P$ is
      $$\norm{\Sigma_{(r_1,\dots,r_n):\Um(R)}P^{D(r_1,\dots,r_n)}}$$
    \end{lemma}
    
    \begin{proof}
      If $Q$ is a modal proposition and $P\rightarrow Q$ we have
      $$\norm{\Sigma_{(r_1,\dots,r_n):\Um(R)}P^{D(r_1,\dots,r_n)}}\rightarrow Q$$
      since
      $P^{D(r_1,\dots,r_n)}\rightarrow Q^{D(r_1,\dots,r_n)}$ and $Q^{D(r_1,\dots,r_n)}\rightarrow Q$.
      It is thus enough to show that
      $$P_0 = \norm{\Sigma_{(r_1,\dots,r_n):\Um(R)}P^{D(r_1,\dots,r_n)}}$$
      is modal.
      If $s_1,\dots,s_m$ is in $\Um(R)$ we show $P_0^{D(s_1,\dots,s_m)}\rightarrow P_0$. This follows
      from $\Um(R)^{D(r)} = \Um(R_r)$, Lemma \ref{Um}.
    \end{proof}

    \begin{proposition}\label{norm}
      For any modal type $T$, the proposition $\norm{T}_S$ is
      $$\norm{\Sigma_{(r_1,\dots,r_n):\Um(R)}T^{D(r_1)}\times\dots\times T^{D(r_n)}}$$
    \end{proposition}
    
    \begin{proof}
      It follows from Lemma \ref{propsheaf} that the proposition $\norm{T}_S$ is
      $$\norm{\Sigma_{(r_1,\dots,r_n):\Um(R)}\norm{T}^{D(r_1,\dots,r_n)}} = \norm{\Sigma_{(r_1,\dots,r_n):\Um(R)}\norm{T}^{D(r_1)}\times\dots\times\norm{T}^{D(r_n)}}$$
      and we get the result using the fact that choice holds for each $D(r_i)$, so that
      \[\norm{T}^{D(r_1)}\times\dots\times\norm{T}^{D(r_n)} = \norm{T^{D(r_1)}}\times\dots\times\norm{T^{D(r_n)}} =
        \norm{T^{D(r_1)}\times\dots\times T^{D(r_n)}}\]
    \end{proof}
    
    \begin{proposition}
      In the sheaf model, $R$ is a local ring.
    \end{proposition}

    \begin{proof}
      This follows from \Cref{norm} and Lemma \ref{Um}.
    \end{proof}

    \begin{lemma}\label{localfp}
      If $A$ is a $R$-algebra which is modal and there exists $r_1,\dots,r_n$ in $\Um(R)$ such that each
      $A^{D(r_i)}$ is a f.p.\ $R_{r_i}$-algebra, then $A$ is a f.p.\ $R$-algebra.
    \end{lemma}
    
    \begin{proof}
      Using the local-global principles presented in \cite{lombardi-quitte}, we can patch together the f.p.\ $R_{r_i}$-algebra
      to a global f.p.\ $R$-algebra. This f.p.\ $R$-algebra is modal by Proposition \ref{modal}, and is locally equal to $A$
      and hence equal to $A$ since $A$ is modal.
    \end{proof}

    \begin{corollary}
      The type of f.p.\ $R$-algebras is modal and is the type of f.p.\ $R$-algebras in the sheaf model.
    \end{corollary}

    \begin{proof}
          For any $R$-algebra $A$, we can form a type $\Phi(n,m,A)$ expressing that $A$ has a presentation for some $v:Pr_{n,m}(R)$,
    as the type stating that there is some map $\alpha:R[X_1,\dots,X_n]\rightarrow A$ and that $(A,\alpha)$ is universal such that
    $\alpha$ is $0$ on all elements of $v$. We can also look at this type $\Phi(n,m,A)_S$ in the sheaf model. Using the translation
    from \cite{Quirin16,CRS21}, we see that the type $\Phi(\eta n,\eta m,A)_S$ is exactly the type stating that $A$ is presented by
    some $v:Pr_{n,m}(A)$ among the modal $R$-algebras. This is actually equivalent to $\Phi(n,m,A)$ since any f.p. $R$-algebra is modal.

     If $A$ is a modal $R$-algebra which is f.p. in the sense of the sheaf model, this means that we have
     $$\norm{\Sigma_{n:\nats_S}\Sigma_{m:\nats_S}\Phi(n,m,A)_S}_S$$
     This is equivalent to
     $$\norm{\Sigma_{n:\nats}\Sigma_{m:\nats}\Phi(\eta n,\eta m,A)_S}_S$$
     which in turn is equivalent to
     $$\norm{\Sigma_{n:\nats}\Sigma_{m:\nats}\Phi(n,m,A)}_S$$
     Using Lemma \ref{localfp} and Proposition \ref{norm}, this is equivalent to $\norm{\Sigma_{n:\nats}\Sigma_{m:\nats}\Phi(n,m,A)}$.
    \end{proof}

     Note that the type of f.p. $R$-algebra is universe independent.

    \begin{proposition}
      For any f.p.\ $R$-algebra $A$, the type $\Spec(A)$ is modal and satisfies the axiom of Zariski local choice in
      the sheaf model.
    \end{proposition}
    
    \begin{proof}
      Let $P(x)$ be a family of types over $x:\Spec(A)$ and assume $\Pi_{x:\Spec(A)}\norm{P(x)}_S$. By Proposition \ref{norm},
      this means $\Pi_{x:\Spec(A)}\norm{\Sigma_{(r_1,\dots,r_n):Um}P(x)^{D(r_1)}\times\dots\times P(x)^{D(r_n)}}$. The result follows
      then from choice over $\Spec(A)$ and Lemma \ref{Um}.
    \end{proof}      


    \subsection{Presheaf models of univalence}

    We recall first how to build presheaf models of univalence \cite{CCHM,survey},
    and presheaf models satisfying the 3 axioms of the previous section.

The constructive models of univalence are presheaf models parametrised by an interval object $\II$
(presheaf with two global distinct elements $0$ and $1$ and which is tiny) and a classifier object
$\Phi$ for cofibrations. The model is then obtained as an internal model of type theory inside the
presheaf model. For this, we define $C:U\rightarrow U$, uniform in the universe $U$, operation
closed by dependent products, sums and such that $C(\Sigma_{X:U}X)$ holds. It further satisfies, for $A:U^{\II}$, the transport principle
$$
(\Pi_{i:\II}C(Ai))\rightarrow (A0\rightarrow A1)
$$
We get then a model of univalence by interpreting a type as a presheaf $A$ together with an element
of $C(A)$.

 This is over a base category $\BB$.
 
 If we have another category $\CC$, we automatically get a new model of univalent type theory by
 changing $\BB$ to $\BB\times\CC$.

 A particular case is if $\CC$ is the opposite of the category of f.p.\@ $k$-algebras, where $k$ is a
 fixed commutative ring.

 We have the presheaf $R$ defined by $R(J,A) = Hom(k[X],A)$ where $J$ object of $\BB$ and $A$ object of $\CC$.

  The presheaf $\Gm$ is defined by $\Gm(J,A) = Hom(k[X,1/X],A) = A^{\times}$, the set of invertible elements of $A$.

\subsection{Propositional truncation}

    We start by giving a simpler interpretation of propositional truncation. This will simplify
    the proof of the validity of choice in the presheaf model.

    We work in the presheaf model over a base category $\BB$ which interprets univalent type theory,
    with a presheaf $\Phi$ of cofibrations. The interpretation of the propositional
    truncation $\norm{T}$ {\em does not} require the use of the interval $\II$.

    We recall that in the models, to be contractible can be formulated as having an operation
    $\ext(\psi,v)$ which extends any partial element $v$ of extent $\psi$ to a total element.

    The (new) remark is then that to be a (h)proposition can be formulated as having instead
    an operation $\ext(u,\psi,v)$ which, now {\em given}
    an element $u$, extends any partial element $v$ of extent $\psi$ to a total element.

\medskip    

Propositional truncation is defined as follows. An element of $\norm{T}$ is either of the form
$\inc(a)$ with $a$ in $T$, or of the form $\ext(u,\psi,v)$ where $u$ is in $\norm{T}$ and $\psi$
in $\Phi$ and $v$ a partial element of extent $\psi$.

In this definition, the special constructor $\ext$ is a ``constructor with restrictions'' which
satisfies $\ext(u,\psi,v) = v$ on the extent $\psi$ \cite{CoquandHM18}.

\subsection{Choice}

We prove choice in the presheaf model: if $A$ is a f.p.\@ algebra over $R$ then we have a map
$$
l:(\Pi_{x:\Spec(A)}\norm{P})\rightarrow \norm{\Pi_{x:\Spec(A)}P}
$$

For defining the map $l$, we define $l(v)$ by induction on $v$.
The element $v$ is in $(\Pi_{x:\Spec(A)}\norm{P})(B)$, which can be seen as
an element of $\norm{P}(A)$. If it is $\inc(u)$ we associate $\inc(u)$ and 
if it is $\ext(u,\psi,v)$ the image is $\ext(l(u),\psi,l(v))$.

\subsection{$1$-topos model}

For any small category $\CC$ we can form the presheaf model of type theory over the base category $\CC$ \cite{hofmann,huber-phd-thesis}.

\medskip

We look at the special case where $\CC$ is the opposite of the category of finitely presented $k$-algebras for a fixed
ring $k$.

    In this model we have a presheaf $R(A) = Hom(k[X],A)$ which has a ring structure.

    In the {\em presheaf} model, we can check that we have $\neg\neg (0=_R 1)$. Indeed, at any stage $A$ we have
    a map $\alpha:A\rightarrow 0$ to the trivial f.p. algebra $0$, and $0 =_R 1$ is valid at the stage $0$.

    The previous internal description of the sheaf model applies as well in the $1$-topos setting.

    \medskip

    However the type of modal types in a given universe is not modal in this $1$-topos setting. This problem can actually be seen as a
    motivation for introducing the notion of stacks, and is solved when we start from a constructive model of univalence.

    \subsection{Some properties of the sheaf model}

    \subsubsection{Quasi-coherence}

A module $M$ in the sheaf model defined at stage $A$, where $A$ is a f.p.\@ $k$-algebra, is given by a sheaf over the category
of elements of $A$. It is thus given by a family of modules $M(B,\alpha)$, for $\alpha:A\rightarrow B$, and restriction maps
$M(B,\alpha)\rightarrow M(C,\gamma\alpha)$ for $\gamma:B\rightarrow C$. In general this family is not determined by
its value $M_A = M(A,\idd_A)$ at $A,\idd_A$. The next proposition expresses when this is the case in an internal way
(this characterisation is due to Blechschmidt \cite{ingo-thesis}).

\begin{proposition}
  $M$ is internally quasi-coherent\footnote{In the sense that the canonical map $M\otimes A\rightarrow M^{\Spec(A)}$ is an isomorphism for any
  f.p. $R$-algebra $A$.} iff we have $M(B,\alpha) = M_A\otimes_A B$ and the restriction map for
  $\gamma:B\rightarrow C$ is $M_A\otimes_A\gamma$.
\end{proposition}

    \subsubsection{Projective space}

We have defined $\bP^n$ to be the set of lines in $V = R^{n+1}$, so we have
$$
\bP^n ~=~ \Sigma_{L:V\rightarrow \Omega}[\exists_{v:V}\neg (v = 0)\wedge L = R v]
$$
The following was noticed in \cite{kockreyes}.

\begin{proposition}
  $\bP^n(A)$ is the set of submodules of $A^{n+1}$ factor direct in $A^{n+1}$ and of rank $1$.
\end{proposition}

\begin{proof}
  $\bP^n$ is the set of pairs $L,0$ where $L:\Omega^V(A)$ satisfies the proposition $\exists_{v:V}\neg (v = 0)\wedge L = Rv$ at stage
  $A$. This condition implies that $L$ is a quasicoherent submodule of $R^{n+1}$ defined at stage $A$.
  It is thus determined by its value $L(A,\idd_A) = L_A$.

  Furthermore, the condition also implies that $L_A$ is locally free of rank $1$. By local-global principle \cite{lombardi-quitte},
  $L_A$ is finitely generated. We can then apply Theorem 5.14 of
  \cite{lombardi-quitte} to deduce that $L_A$ is factor direct in $A^{n+1}$ and of rank $1$.
\end{proof}

One point in this argument was to notice that the condition
$$
\exists_{v:V}\neg (v = 0)\wedge L = R v
$$
implies that $L$ is quasi-coherent. This would be direct in presence of univalence, since we would have then $L = R$ as a $R$-module
and $R$ is quasi-coherent. But it can also be proved without univalence by transport along isomorphism: a $R$-module which is
isomorphic to a quasi-coherent module is itself quasi-coherent.

\subsection{Global sections and Zariski global choice}

We let $\Box T$ the type of global sections of a globally defined sheaf $T$.
If $c = r_1,\dots,r_n$ is in $\Um(R)$ we let $\Box_c T$ be the type $\Box T^{D(r_1)}\times\dots\times\Box T^{D(r_n)}$.

Using these notations, we can state the principle of Zariski global choice
$$
(\Box \norm{T})\leftrightarrow \norm{\Sigma_{c:\Um(k)}\Box_c T}
$$

This principle is valid in the present model.

Using this principle, we can show that $\Box K(\Gm,1)$ is equal to the type of projective modules of rank $1$ over $k$
and that each $\Box K(R,n)$ for $n>0$ is contractible.
                                                                                  

\appendix

\section{Negative results}
Here we collect some results of
the theory developed from the axioms
(\axiomref{loc}), (\axiomref{sqc}) and (\axiomref{Z-choice})
that are of a negative nature
and primarily serve the purpose of counterexamples.

We adopt the following definition from
\cite[Section IV.8]{lombardi-quitte}.

\begin{definition}%
  \label{zero-dimensional-ring}
  A ring $A$ is \notion{zero-dimensional}
  if for all $x : A$
  there exists $a : A$ and $k : \N$
  such that $x^k = a x^{k + 1}$.
\end{definition}

\begin{lemma}[using \axiomref{loc}, \axiomref{sqc}, \axiomref{Z-choice}]%
  \label{R-not-zero-dimensional}
  The ring $R$ is not zero-dimensional.
\end{lemma}

\begin{proof}
  Assume that $R$ is zero-dimensional,
  so for every $f : R$ there merely is some $k : \N$ with $f^k \in (f^{k + 1})$.
  We note that $R = \A^1$ is an affine scheme and
  that if $f^k \in (f^{k + 1})$,
  then we also have $f^{k'} \in (f^{k' + 1})$ for every $k' \geq k$.
  This means that we can apply \Cref{strengthened-boundedness}
  and merely obtain a number $K : \N$
  such that $f^K \in (f^{K + 1})$ for all $f : R$.
  In particular, $f^{K + 1} = 0$ implies $f^K = 0$,
  so the canonical map
  $\Spec R[X]/(X^K) \to \Spec R[X]/(X^{K + 1})$
  is a bijection.
  But this is a contradiction,
  since the homomorphism $R[X]/(X^{K + 1}) \to R[X]/(X^K)$
  is not an isomorphism.
\end{proof}

\begin{example}[using \axiomref{loc}, \axiomref{sqc}, \axiomref{Z-choice}]%
  \label{non-existence-of-roots}
  It is not the case that
  every monic polynomial $f : R[X]$ with $\deg f \geq 1$ has a root.
  More specifically,
  if $U \subseteq \A^1$ is an open subset
  with the property that
  the polynomial $X^2 - a : R[X]$ merely has a root
  for every $a : U$,
  then $U = \emptyset$.
\end{example}

\begin{proof}
  Let $U \subseteq \A^1$ be as in the statement.
  Since we want to show $U = \emptyset$,
  we can assume a given element $a_0 : U$
  and now have to derive a contradiction.
  By \axiomref{Z-choice},
  there exists in particular a standard open $D(f) \subseteq \A^1$
  with $a_0 \in D(f)$
  and a function $g : D(f) \to R$
  such that ${(g(x))}^2 = x$ for all $x : D(f)$.
  By \axiomref{sqc},
  this corresponds to an element $\frac{p}{f^n} : R[X]_f$
  with ${(\frac{p}{f^n})}^2 = X : R[X]_f$.
  We use \Cref{polynomial-with-regular-value-is-regular}
  together with the fact that $f(a_0)$ is invertible
  to get that $f : R[X]$ is regular,
  and therefore $p^2 = f^{2n}X : R[X]$.
  Considering this equation over $R^{\mathrm{red}} = R/\sqrt{(0)}$ instead,
  we can show by induction that all coefficients of $p$ and of $f^n$ are nilpotent,
  which contradicts the invertibility of $f(a_0)$.
\end{proof}

\begin{remark}
  \Cref{non-existence-of-roots} shows that
  the axioms we are using here
  are incompatible with a natural axiom that is true
  for the structure sheaf of the big étale topos,
  namely that $R$ admits roots for unramifiable monic polynomials.
  The polynomial $X^2 - a$ is even separable for invertible $a$,
  assuming that $2$ is invertible in $R$.
  To get rid of this last assumption,
  we can use the fact that either $2$ or $3$ is invertible in the local ring $R$
  and observe that the proof of \Cref{non-existence-of-roots}
  works just the same for $X^3 - a$.
\end{remark}

We now give two different proofs that not all $R$-modules are weakly quasi-coherent
in the sense of \Cref{weakly-quasi-coherent-module}.
The first shows that the map
\[ M_f \to M^{D(f)} \]
is not always surjective,
the second shows that it is not always injective.

\begin{proposition}[using \axiomref{loc}, \axiomref{sqc}, \axiomref{Z-choice}]%
  \label{RN-non-wqc}
  The $R$-module $R^\N$ is not weakly quasi-coherent
  (in the sense of \Cref{weakly-quasi-coherent-module}).
\end{proposition}

\begin{proof}
  For $f : R$,
  we have ${(R^{\N})}^{D(f)} = {(R^{D(f)})}^\N = {(R_f)}^\N$,
  so the question is whether the canonical map
  \[ {(R^\N)}_f \to {(R_f)}^\N \]
  is an equivalence.
  If it is,
  for a fixed $f : R$,
  then the sequence $(1, \frac{1}{f}, \frac{1}{f^2}, \dots)$
  has a preimage,
  so there is an $n : \N$ such that
  for all $k : \N$,
  $\frac{a_k}{f^n} = \frac{1}{f^k}$ in $R_f$
  for some $a_k : R$.
  In particular, $\frac{a_{n+1}}{f^n} = \frac{1}{f^{n+1}}$ in $R_f$
  and therefore $a_{n+1} f^{n+1+\ell} = f^{n+\ell}$ in $R$ for some $\ell : \N$.
  This shows that $R$ is zero-dimensional
  (\Cref{zero-dimensional-ring})
  if $R^\N$ is weakly quasi-coherent.
  So we are done by \Cref{R-not-zero-dimensional}.
\end{proof}

\begin{proposition}[using \axiomref{loc}, \axiomref{sqc}, \axiomref{Z-choice}]%
  \label{non-wqc-module-family}
  The implication
  \[ M^{D(f)} = 0 \quad\Rightarrow\quad M_f = 0 \]
  does not hold for all $R$-modules $M$ and $f : R$.
  In particular,
  the map $M_f \to M^{D(f)}$ from \Cref{weakly-quasi-coherent-module}
  is not always injective.
\end{proposition}

\begin{proof}
  Assume that the implication always holds.
  We construct a family of $R$-modules,
  parametrized by the elements of $R$,
  and deduce a contradiction from the assumption
  applied to the $R$-modules in this family.

  Given an element $f : R$,
  the $R$-module we want to consider is
  the countable product
  \[ M(f) \colonequiv \prod_{n : \N} R/(f^n) \rlap{.} \]
  If $f \neq 0$ then $M(f) = 0$
  (using \Cref{non-zero-invertible}).
  This implies that the $R$-module $M(f)^{f \neq 0}$
  is trivial:
  any function $f \neq 0 \to M(f)$ can only assign the value $0$
  to any of the at most one witnesses of $f \neq 0$.
  By assumption, this implies that $M(f)_f$ is also trivial.
  Noting that
  $M(f)$ is not only an $R$-module
  but even an $R$-algebra in a natural way,
  we have
  \begin{align*}
    M(f)_f = 0
    &\;\Leftrightarrow\;
    \exists k : \N.\; \text{$f^k = 0$ in $M(f)$} \\
    &\;\Leftrightarrow\;
    \exists k : \N.\; \forall n : \N.\; f^k \in (f^n) \subseteq R \\
    &\;\Leftrightarrow\;
    \exists k : \N.\; f^k \in (f^{k + 1}) \subseteq R
    \rlap{.}
  \end{align*}

  In summary,
  our assumption implies that the ring $R$ is zero-dimensional
  (in the sense of \Cref{zero-dimensional-ring}).
  But this is not the case,
  as we saw in \Cref{R-not-zero-dimensional}.
\end{proof}

\begin{example}[using \axiomref{loc}, \axiomref{sqc}]
  It is not the case that
  for any pair of lines $L, L' \subseteq \bP^2$,
  the $R$-algebra $R^{L \cap L'}$ is
  as an $R$-module free of rank $1$.
\end{example}

\begin{proof}
  The $R$-algebra $R^{L \cap L'}$ is free of rank $1$
  if and only if the structure homomorphism
  $\varphi : R \to R^{L \cap L'}$ is bijective.
  We will show that it is not even always injective.

  Consider the lines
  \[ L = \{\, [x : y : z] : \bP^2 \mid z = 0 \,\} \]
  and
  \[ L' = \{\, [x : y : z] : \bP^2 \mid \varepsilon x + \delta y + z = 0 \,\}
     \rlap{,} \]
  where $\varepsilon$ and $\delta$ are elements of $R$
  with $\varepsilon^2 = \delta^2 = 0$.
  Consider the element $\varphi(\epsilon \delta) : R^{L \cap L'}$,
  which is the constant function $L \cap L' \to R$
  with value $\varepsilon \delta$.
  For any point $[x : y : z] : L \cap L'$,
  we have $z = 0$ and $\varepsilon x + \delta y = 0$.
  But also, by definition of $\bP^3$,
  we have $(x, y, z) \neq 0 : R^3$,
  so one of $x, y$ must be invertible.
  This implies $\delta \divides \varepsilon$ or $\varepsilon \divides \delta$,
  and in both cases we can conclude $\varepsilon \delta = 0$.
  Thus, $\varphi(\epsilon \delta) = 0 : R^{L \cap L'}$.

  If $\varphi$ was always injective
  then this would imply $\varepsilon \delta = 0$
  for any $\varepsilon, \delta : R$
  with $\varepsilon^2 = \delta^2 = 0$.
  In other words, the inclusion
  \[ \Spec R[X, Y]/(X^2, Y^2, XY) \hookrightarrow \Spec R[X, Y]/(X^2, Y^2) \]
  would be a bijection.
  But the corresponding $R$-algebra homomorphism is not an isomorphism.
\end{proof}

\printindex

\printbibliography

\end{document}